\documentclass[11pt]{article}

\usepackage[nottoc,notlot,notlof]{tocbibind}
\usepackage{amsmath,amssymb,amsfonts,amsthm}
\usepackage{latexsym}
\usepackage{tikz}
\usepackage[normalem]{ulem}
\usetikzlibrary{automata,positioning}
\usetikzlibrary{chains,fit,shapes}
\usetikzlibrary{calc}
\usetikzlibrary{arrows}
\usetikzlibrary{positioning,calc}
\usetikzlibrary{graphs}
\usetikzlibrary{graphs.standard}
\usetikzlibrary{arrows,decorations.markings}
\usepackage{multicol}
\usepackage{xcolor}
\newtheorem{theorem}{Theorem}
\newtheorem{definition}[theorem]{Definition}
\newtheorem{lemma}[theorem]{Lemma}
\newtheorem{remark}[theorem]{Remark}
\newtheorem{proposition}[theorem]{Proposition}
\newtheorem{corollary}[theorem]{Corollary}

\newtheorem{example}[theorem]{Example}

\newcommand{\IST}{{\rm IST}}

\title{Semi-transitivity of directed split graphs generated by morphisms}

\author{Kittitat Iamthong\footnote{Department of Mathematics and Statistics, University of Strathclyde, 26 Richmond Street Glasgow G1 1XH, United Kingdom} \ and Sergey Kitaev\footnotemark[1] \\ 
{{\footnotesize kittitat.iamthong@strath.ac.uk},\,{\footnotesize sergey.kitaev@strath.ac.uk}}}

\begin{document}  

\maketitle

\abstract{A directed graph is semi-transitive if and only if it is acyclic and for any directed path $u_1\rightarrow u_2\rightarrow \cdots \rightarrow u_t$, $t \geq 2$, either there is no edge from $u_1$ to $u_t$ or all edges $u_i\rightarrow u_j$ exist for $1 \leq i < j \leq t$. 

In this paper, we study semi-transitivity of families of directed split graphs obtained by iterations of morphisms applied to the adjacency matrices and giving in the limit infinite directed split graphs. A split graph is a graph in which the vertices can be partitioned into a clique and an independent set. We fully classify semi-transitive infinite directed split graphs when a morphism in question can involve any $n\times m$ matrices over $\{-1,0,1\}$ with a single natural condition. 
}

\section{Introduction}\label{Sec-intro}

The notion of a semi-transitive orientation of a graph was introduced by Halld\'orsson et al.\ in \cite{HKP11} (also see \cite{HKP16}) as means to completely characterize so-called word-representable graphs \cite{K17,KL15}:\ A graph is word-representable if and only if it admits a semi-transitive orientation. Word-representable graphs, and thus semi-transitive graphs (i.e.\ semi-transitively orientable graphs), generalize several important classes of graphs, e.g.\ circle graphs, 3-colorable graphs and comparability graphs. Semi-transitive orientations are also interesting in their own right as a generalization of transitive orientations. 

{\em Split graphs} \cite{FH77} are graphs in which the vertices can be partitioned into a clique and an independent set.  The study of split graphs attracted much attention in the literature (e.g.\ see \cite{CJKS2020} and references therein). Related to our context, the study of semi-transitive orientability of split graphs was initiated in \cite{CKS19,KLMW17}, where certain subclasses of semi-transitive split graphs were characterized in terms of forbidden subgraphs. Also, split graphs were instrumental in \cite{CKS19} to solve a 10 year old open problem in the theory of word-representable graphs.

In a recent work \cite{I2021}, the first author of this paper extended the studies in \cite{CKS19,KLMW17} by characterizing semi-transitive split graphs in terms of permutations of columns of the adjacency matrices. Moreover, \cite{I2021} studies semi-transitivity of split graphs obtained by iterations of morphisms applied to the adjacency matrices, and thus giving yet another link to combinatorics on words \cite{Loth} (the original link comes from the definition of a word-representable graph). A number of general theorems and a complete classification of semi-transitive orientability in the case of morphisms defined by $2\times2$ matrices are given in \cite{I2021}.

In this paper, we study families of {\em directed} split graphs obtained by iterations of morphisms (involving three matrices $A,B,C$) applied to the adjacency matrices and giving as the limit infinite directed split graphs. For each of such a family we ask the question on whether all graphs in the family are oriented semi-transitively (i.e.\ are semi-transitive) or a finite iteration $k$ of the morphism produces a non-semi-transitive orientation (which will stay non-semi-transitive for all iterations $>k$). In the former case, we say that the infinite split graph's index of semi-transitivity is $\infty$ (denoted $\IST(A,B,C)= \infty$), and in the latter case it is $k$ (assuming $k$ is minimal possible).   

The novelty of our paper is in the study of directed graphs in connection to semi-transitive orientations (as opposed to undirected graphs in the long list of relevant research papers cited in \cite{K17,KL15}), and in that we offer a way to generate interesting (from semi-transitivity point of view) families of directed split graphs using adjacency matrices and iterations of morphisms. Our research will contribute to improving further known algorithms to recognise semi-transitive orientations (on directed split graphs and beyond).  It comes somewhat as a surprise that we were able to completely classify infinite directed split graphs with the index of semi-transitivity $\infty$, where  morphisms in question involve almost arbitrary $n\times m$ matrices over $\{-1,0,1\}$ as opposed to, say, $2\times2$ matrices in \cite{I2021} (in a different context though); the only natural condition, to ensure that our definitions work, is that $A$ has a 0. Our classification is done via several results depending on the structures of matrices $A,B,C$ in question, and it is summarised in the diagram in Figure~\ref{diagram}. Following the  diagram, one can easily determine whether $\IST(A,B,C)= \infty$ for any given $A,B,C$.  

\begin{figure}
	\begin{center}
		\includegraphics[scale=0.6]{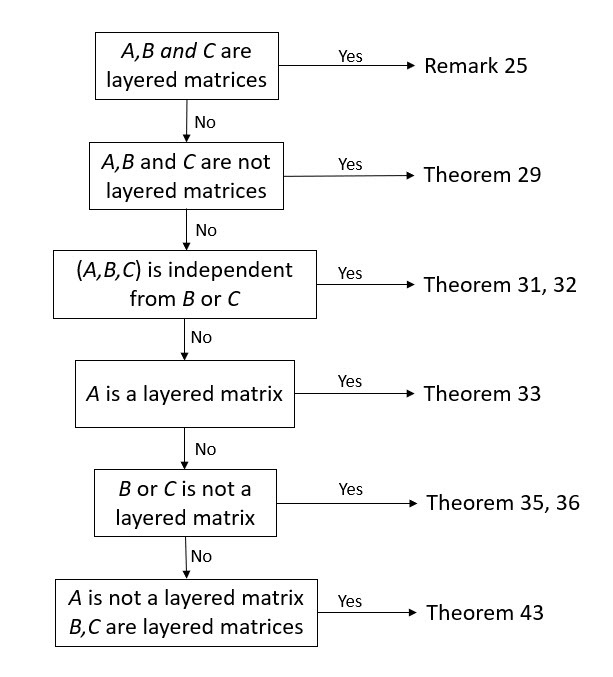}
	\end{center}
	\caption{A guide to the classification results where $A$ is assumed to have a $0$ (a natural condition to ensure that our definitions work). For example, if none of $A,B,C$ is a layered matrix then Theorem~\ref{Thm-ABC-not-layered} is to be applied. 
	}\label{diagram}
\end{figure}

\section{Preliminaries} \label{Sec-preliminaries}

\subsection{Semi-transitive orientations and split graphs}

Graphs in this paper have no loops or multiple edges. Any split graph $S_n$ on $n$ vertices can be partitioned into a {\em maximal} clique $K_m$ and an independent set $E_{n-m}$, and we write $S_n =(E_{n-m},K_m)$.

A directed graph is  oriented {\em semi-transitively} if and only if it is acyclic and for any directed path $u_1\rightarrow u_2\rightarrow \cdots \rightarrow u_t$, $t \geq 2$, either there is no edge from $u_1$ to $u_t$ or all edges $u_i\rightarrow u_j$ exist for $1 \leq i < j \leq t$. Graphs admitting semi-transitive orientations are {\em semi-transitive}.

In this paper, we will need the following results on semi-transitive orientations and split graphs, where a {\em source} (resp., {\em sink}) is a vertex of in-degree (resp., out-degree) 0.

\begin{lemma}[\cite{KLMW17}]\label{lem-tran-orie} Let $K_m$ be a clique in a graph $G$. Then any acyclic orientation of $G$ induces a transitive orientation on $K_m$ (where the presence of edges $u\rightarrow v$ and $v\rightarrow z$ implies the presence of the edge $u\rightarrow z$). In particular, any semi-transitive orientation of $G$ induces a transitive orientation on $K_m$. In either case, the orientation induced on $K_m$ contains a single source and a single sink. \end{lemma}

\begin{theorem}[\cite{KLMW17}]\label{semi-tran-groups} Any semi-transitive orientation of a split graph $S_n=(E_{n-m},K_m)$ subdivides the set of all vertices in $E_{n-m}$ into three, possibly empty, groups corresponding to each of the following types (also shown schematically in Figure~\ref{3-groups}), where $\vec{P}= p_1\rightarrow\cdots\rightarrow p_m$ is the longest directed path in $K_m$: 
	\begin{itemize}
		\item A vertex in $E_{n-m}$ is of {\em type A} if it is a source and is connected to all vertices in $\{p_i,p_{i+1},\ldots, p_j\}$ for some $1\leq i\leq j\leq m$;
		\item A vertex in $E_{n-m}$ is of {\em type B} if it is a sink and is connected to all vertices in $\{p_i,p_{i+1},\ldots, p_j\}$ for some $1\leq i\leq j\leq m$; 
		\item A vertex $v\in E_{n-m}$ is of {\em type C} if there is an edge $x\rightarrow v$ for each $x\in I_v=\{p_1,p_2,\ldots, p_i\}$ and there is an edge $v\rightarrow y$ for each $y\in O_v=\{p_j,p_{j+1},\ldots, p_m\}$ for some $1\leq i< j\leq m$.
	\end{itemize} 
\end{theorem}

\begin{figure}
	\begin{center}
		
		\begin{tabular}{ccc}
			
			\begin{tikzpicture}[->,>=stealth', shorten >=1pt,node distance=0.5cm,auto,main node/.style={fill,circle,draw,inner sep=0pt,minimum size=5pt}]
			
			\node[main node] (1) {};
			\node[main node] (2) [right of=1,xshift=5mm] {};
			\node[main node] (3) [above of=2] {};
			\node (4) [above of=3] {};
			\node[main node] (5) [below of=2] {};
			\node (6) [below of=5] {};
			
			\node (7) [above of=4,yshift=-4mm,xshift=1mm] {};
			\node (8) [above right of=7,xshift=6mm]{};
			
			\node (10) [below of=6,yshift=4mm,xshift=1mm] {};
			\node (11) [below right of=10,xshift=6mm]{};
			
			\node (11) [below left of=5,yshift=-5mm]{type A};
			
			\path
			(4) edge (3)
			(3) edge (2)
			(2) edge (5)
			(5) edge (6);
			
			\path
			(1) edge (2)
			(1) edge (3)
			(1) edge (5);
			
			\end{tikzpicture}
			
			&
			\begin{tikzpicture}[->,>=stealth', shorten >=1pt,node distance=0.5cm,auto,main node/.style={fill,circle,draw,inner sep=0pt,minimum size=5pt}]
			
			\node[main node] (1) {};
			\node[main node] (2) [right of=1,xshift=5mm] {};
			\node[main node] (3) [above of=2] {};
			\node (4) [above of=3] {};
			\node[main node] (5) [below of=2] {};
			\node (6) [below of=5] {};
			
			\node (7) [above of=4,yshift=-4mm,xshift=1mm] {};
			\node (8) [above right of=7,xshift=6mm]{};
			
			\node (10) [below of=6,yshift=4mm,xshift=1mm] {};
			\node (11) [below right of=10,xshift=6mm]{};
			
			\node (11) [below left of=5,yshift=-5mm]{type B};
			
			\path
			(4) edge (3)
			(3) edge (2)
			(2) edge (5)
			(5) edge (6);
			
			\path
			(2) edge (1)
			(3) edge (1)
			(5) edge (1);
			
			\end{tikzpicture}
			
			&
			
			\begin{tikzpicture}[->,>=stealth', shorten >=1pt,node distance=0.5cm,auto,main node/.style={fill,circle,draw,inner sep=0pt,minimum size=5pt}]
			
			\node[main node] (1) {};
			\node[main node] (2) [below right of=1,xshift=5mm] {};
			\node[main node] (3) [above of=2] {};
			\node[main node] (4) [above of=3] {};
			\node[main node] (5) [below of=2] {};
			\node[main node] (6) [below of=5] {};
			\node[main node] (13) [above of=4] {};
			\node[main node] (14) [below of=5] {};
			
			\node (7) [above of=13,yshift=-5mm,xshift=1mm] {};
			\node (8) [right of=7,xshift=6mm]{};
			\node (9) [right of=8]{source};
			
			\node (10) [below of=14,yshift=5mm,xshift=1mm] {};
			\node (11) [right of=10,xshift=6mm]{};
			\node (12) [right of=11,xshift=-1mm]{sink};
			
			\node [below of=14]{type C};
			
			\path
			(8) edge (7)
			(11) edge (10);
			
			\path
			(13) edge (1)
			(4) edge (1)
			(1) edge (14)
			(1) edge (5)
			(4) edge (3)
			(3) edge (2)
			(2) edge (5)
			(5) edge (6)
			(13) edge (4)
			(5) edge (14);
			
			\end{tikzpicture}
			
		\end{tabular}
		
		\caption{\label{3-groups} Three types of vertices in $E_{n-m}$ in a semi-transitive orientation of $(E_{n-m},K_m)$. The vertical oriented paths are a schematic way to show (parts of) $\vec{P}$}
	\end{center}
\end{figure}
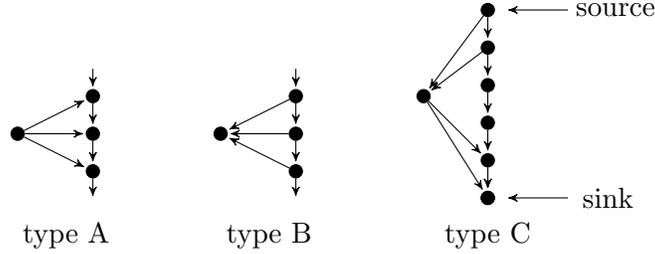

\begin{theorem}[\cite{KLMW17}]\label{relative-order} Let $S_n=(E_{n-m},K_m)$ be oriented semi-transitively with $\vec{P}= p_1\rightarrow\cdots\rightarrow p_m$. For a vertex $x\in E_{n-m}$ of type C,
	there is no vertex $y\in E_{n-m}$ of type A or B, which is connected to both $p_{|I_x|}$ and $p_{m-|O_x|+1}$. Also, there is no vertex $y\in E_{n-m}$ of type C such that either $I_y$, or $O_y$ contains both $p_{|I_x|}$ and $p_{m-|O_x|+1}$.
\end{theorem}

\begin{theorem}[\cite{KLMW17}]\label{main-orientation} An orientation of a split graph $S_n=(E_{n-m},K_m)$ is semi-transitive if and only if 
	\begin{itemize} 
		\item[{\em (i)}] $K_m$ is oriented transitively;
		\item[{\em (ii)}] each vertex in $E_{n-m}$ is of one of the three types in Theorem~\ref{semi-tran-groups};
		\item[{\em (iii)}] the restrictions in Theorem~\ref{relative-order} are satisfied. 
\end{itemize}\end{theorem}


\subsection{Directed split graphs}

A directed graph is {\em semi-transitive} if its orientation is semi-transitive. The {\em adjacency matrix} $A=[a_{ij}]$ of a directed graph on $n$ vertices is a binary matrix such that $a_{ij}=1$ if $j\rightarrow i$ is an edge, and $a_{ij}=0$ otherwise.  Let $L(A)=[\ell_{ij}]$ be the $n \times n$ lower triangular matrix such that, for any $i > j$, $$\ell_{ij}= \begin{cases}
	1 &\text{~~if~~}a_{ij}=1,\\
	-1 &\text{~~if~~}a_{ji}=1,\\
	0 &\text{~~otherwise~~}
	\end{cases}$$ and $\ell_{ij}=0$ for any $i \leq j$.

Clearly, there is a one-to-one correspondence between directed graphs of order $n$ and $n \times n$ lower triangular matrices over $\{-1,0,1\}$ with the diagonal elements equal 0. Thus, $L(A)$ can play the role of the adjacency matrix of a directed graph. For $i>j$, the connectivity between vertices $i$ and $j$ is $j \rightarrow i$ if $\ell_{ij}=1$, and is $i \rightarrow j$ if $\ell_{ij}=-1$, and there is no edge if $\ell_{ij}=0$.

\begin{example}\label{ExL(A)} {\em 
		If $A=\left[ {\begin{array}{cccccc}
		0&0&1&0&1&0\\
		1&0&0&0&0&0\\
		0&0&0&1&0&0\\
		0&1&0&0&0&1\\
		0&0&0&0&0&0\\
		1&0&0&0&1&0			\end{array} } \right]$ is an adjacency matrix of a directed graph $G$, then $L(A)=\left[ {\begin{array}{cccccc}
		0&0&0&0&0&0\\
		1&0&0&0&0&0\\
		-1&0&0&0&0&0\\
		0&1&-1&0&0&0\\
		-1&0&0&0&0&0\\
		1&0&0&-1&1&0			\end{array} } \right]$ and the set of edges of $G$ (on 6 vertices) is \{$1\rightarrow 2$, $2\rightarrow 4$, $1\rightarrow 6$, $5\rightarrow 6$, $3\rightarrow 1$, $5\rightarrow 1$, $4\rightarrow 3$, $6\rightarrow 4$\}.
	}
\end{example}

Our interest is in acyclically (without directed cycles) oriented split graphs since only such graphs have a chance to be semi-transitive.  For any acyclically oriented split graph $G$, by Lemma~\ref{lem-tran-orie}, we know that the induced orientation of the maximal clique in $G$ is transitive, so the following notion can be introduced.
\begin{definition}
	An acyclically oriented split graph $G$ with a maximal clique of order $n$ is {\em well-labelled} if the vertex set of $G$ is $V(G)=\{1,2, \ldots, |V(G)|\}$ and the longest directed path in the maximal clique is $1 \rightarrow 2 \rightarrow \cdots \rightarrow n$.
\end{definition} 

Since we can relabel graphs, throughout the paper, W.L.O.G.\ we can assume that any given acyclically oriented split graph is well-labelled. If $A$ is the adjacency matrix for $S=(E_m,K_n)$ (where $K_n$ is maximal) of order $m+n$, then $$L(A) = \begin{bmatrix}
L_n & O_{n,m} \\ M & O_m
\end{bmatrix}$$
for some $m \times n$ matrix $M$, where $O_{n,m}$ and $O_m$ are $n \times m$ and $m \times m$ zero matrices, respectively, and $L_n$ is the $n \times n$ matrix such that all entries strictly below the main diagonal are $1$'s, and all other entries are $0$'s. Hence, every directed split graph with maximal clique of order $n$ and independent set of order $m$ can be represented by an $m \times n$ matrix $M$ appearing in $L(A)$ and recording directed edges between $K_n$ and $E_m$. Thus,   generating a matrix $M$ with entries in $\{-1,0,1\}$, we generate an acyclically oriented split graph.

\begin{definition}\label{defSM}
	Let $M=[m_{ij}]$ be an $m \times n$ matrix such that $m_{ij}\in \{ -1,0,1 \} $ for $1 \leq i \leq m$ and $1 \leq j \leq n$. Define 
	$$S_o(M)=\begin{bmatrix}
	L_n & O_{n,m} \\ M & O_m
	\end{bmatrix}$$ 
where the subscript $o$ stands for ``oriented" and $S$ stands for ``split". We denote the directed split graph corresponding to $S_o(M)$ by $G_o(M)$. 
\end{definition}

\begin{example} \label{ExM} {\em 
		If $M=\left[ {\begin{array}{cccc}
			0&1&0&1\\
			-1&0&-1&-1\\
			0&0&0&1
			\end{array} } \right]$ then $$S_o(M)=\left[ {\begin{array}{cccccccc}
			0&0&0&0&0&0&0\\
			1&0&0&0&0&0&0\\
			1&1&0&0&0&0&0\\
			1&1&1&0&0&0&0\\
			0&1&0&1&0&0&0\\
			-1&0&-1&-1&0&0&0\\
			0&0&0&1&0&0&0
			\end{array} } \right]$$
		is the adjacency matrix of the directed graph $G_0(M)$ shown in Figure~\ref{ExampleSM}.
	}
\end{example}

\begin{figure} 
	\begin{center}
		
		\begin{tikzpicture}[->,>=stealth', shorten >=1pt,node distance=0.5cm,auto,main node/.style={fill,circle,draw,inner sep=0pt,minimum size=5pt}]

		\node[main node] (1) {};
		\node [left of=1, xshift=2.8mm, yshift=0mm] {\tiny{4}};
	\node[main node] (2) 
			[above right of=1, xshift=1.6mm, yshift=2mm] {};
		\node [above of=2, xshift=-0.1mm, yshift=-2.5mm] {\tiny{3}};
			\node[main node] (3) 
			[right of=2, xshift=2mm] {};
		\node [above of=3, xshift=-0.1mm, yshift=-2.5mm] {\tiny{2}};
	\node[main node] (4) 
			[below right of=3, xshift=1.6mm, yshift=-2mm] {};
		\node [right of=4, xshift=-2.8mm, yshift=0mm] {\tiny{1}};
		\node[main node] (5) 
			[below of=1, xshift=2.5mm, yshift=-4mm] {};
		\node [below of=5, xshift=-0.1mm, yshift=2.5mm] {\tiny{5}};
		\node[main node] (6) 
			[right of=5, xshift=1mm] {};
		\node [below of=6, xshift=-0.1mm, yshift=2.5mm] {\tiny{6}};
			\node[main node] (7) 
			[right of=6, xshift=1mm] {};
	\node [below of=7, xshift=-0.1mm, yshift=2.5mm] {\tiny{7}};
	
		\path
		(4) edge (3)
		(4) edge (2)
		(4) edge (1)
		(3) edge (2)
		(3) edge (1)		
		(2) edge (1);
		
		\path
		(5) edge (2)
		(5) edge (4)
		(7) edge (4)
		(1) edge (6)
		(3) edge (6)
		(4) edge (6);
		
		\end{tikzpicture}
		
		\vspace{-4mm}
		\caption{\label{ExampleSM}  The directed split graph $G_o(M)$ given by $S_o(M)$ in Example \ref{ExM}}
	\end{center}
\end{figure}
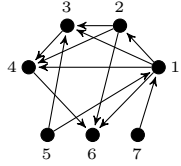

For convenience, we will represent rows of an $m \times n$ matrix $M$ by strings of length $n$. For example, we will represent the three rows of ${\tiny \begin{bmatrix} 1&-1&0&1\\0&1&-1&0\\0&0&0&1 \end{bmatrix}}$ by $1(-1)01$, $01(-1)0$ and $0001$. 

Note that in Definition \ref{defSM}, the maximal clique of $G_o(M)$ is of order $n+1$ if there is a row of the form $11 \cdots 1$ or $(-1)(-1)\cdots (-1)$ in $M$, and the maximal clique is of order $n$  otherwise. In the former case, $G_o(M)$ may not be well-labelled. In the case of $n=1$, the graph $G_o(M)$ is a tree which is always semi-transitive. Thus, throughout this paper, we can assume that $n \geq 2$.

\begin{remark}\label{M=0}
	If $M$ is a zero matrix, then $G_o(M)$ is  semi-transitive as it is a disjoint union of a transitively oriented clique and isolated vertices. 
\end{remark}

In what follows, $x^r$ denotes $xx\cdots x$, where $x\in\{-1,0,1\}$ is repeated $r$ times. 

\begin{lemma} \label{Lem-no0row}
	Let $M := [m_{ij}]_{m \times n}$ be an $m \times n$ matrix over $\{-1,0,1\}$ such that $m_{p1}=m_{p2}=\cdots=m_{pr}=1$ and $m_{p(r+1)}=m_{p(r+2)}=\cdots=m_{pn}=-1$ for some $p \in \{1,2,\ldots,m\}$ and $r \in \{0,1,\ldots,n\}$. If 
	$$N= \begin{bmatrix}
	m_{11}	  &m_{12}	 &\cdots&m_{1r}	   &0	  &m_{1(r+1)}&\cdots&m_{1n}\\	
	m_{21}	  &m_{22}	 &\cdots&m_{2r}	   &0	  &m_{2(r+1)}&\cdots&m_{2n}\\
	\vdots 	  &\vdots	 &      &\vdots    &\vdots&\vdots 	 &      &\vdots \\
	m_{(p-1)1}&m_{(p-1)2}&\cdots&m_{(p-1)r}&0	  &m_{(p-1)(r+1)}&\cdots&m_{(p-1)n}\\
	m_{(p+1)1}&m_{(p+1)2}&\cdots&m_{(p+1)r}&0	  &m_{(p+1)(r+1)}&\cdots&m_{(p+1)n}\\
	\vdots 	  &\vdots	 &      &\vdots    &\vdots&\vdots    &      & \vdots \\
	m_{m1}	  &m_{m2}	 &\cdots&m_{mr}	   &0	  &m_{m(r+1)}&\cdots&m_{mn}
	
\end{bmatrix}$$
is an $(m-1) \times (n+1)$ matrix, then $G_o(M)$ is isomorphic to $G_o(N)$.
\end{lemma}

\begin{proof}
The $p$-th row in $M$, which is $1^r(-1)^{n-r}$, represents the vertex $n+p$ in the independent set connected to all vertices in $K_n = \{1,2, \ldots , n\}$. So $K_n$ is not the maximal clique in $G_o(M)$, but $K_n \cup \{n+p\}$ is the maximal clique. Note that  $\ell \rightarrow n+p$ for every vertex $\ell \in \{1,2,\ldots ,r\}$ and $n+p \rightarrow \ell$ for all  vertex $\ell \in \{r+1,r+2, \ldots, n\}$. We relabel the vertex $n+p$ to be $r+1$ and relabel a vertex $\ell$ to be $\ell+1$ for each $\ell \in \{r+1,r+2,\dots,n+p-1 \}$. The relabelling gives the graph that can be represented by the matrix $S_o(N)$. Hence, $G_o(M)$ is isomorphic to $G_o(N)$.
\end{proof} 

\begin{remark}  \label{remark-no0not1(-1)}
	Let $M$ be an $m \times n$ matrix over $\{-1,0,1\}$. If $a_1a_2\cdots a_n$ is the $p$-th row in $M$ such that $a_q=-1$ and $a_r=1$ for some $1 \leq q < r \leq n$, then $q \rightarrow r \rightarrow n+p \rightarrow q$ forms a cycle in $G_o(M)$. Hence, $G_o(M)$ is not semi-transitive if there is a $1$ occurring to the right of a $-1$ in a row in $M$. Consequently, if there is a row in $M$ such that it has no $0$ and it is not of the form $11\cdots1 (-1)(-1)\cdots(-1)$, then $G_o(M)$ is not semi-transitive.
\end{remark}

Let $M$ be an $m \times n$ matrix over $\{-1,0,1\}$. We can see that the maximal clique of $G_o(M)$ is of order $n$ or $n+1$. Moreover, the maximal clique of $G_o(M)$ is the clique of order $n+1$ if there is a row in $M$ containing no $0$. In this case, the matrix $M$ does not represent only edges between vertices in the maximal clique and vertices in the independent set, but also a vertex in the maximal clique. By Remark~\ref{remark-no0not1(-1)}, we can assume that $M$ does not contain a row which has no $0$ and is not of the form $1^r(-1)^{n-r}$ for some $0 \leq r \leq n$. Hence, if a row of $M$ has no $0$, it must be $1^r(-1)^{n-r}$ for some $1 \leq r \leq n$ for graph $G_o(M)$ to have a chance to be semi-transitive.  Further, if $1^r(-1)^{n-r}$ is a row of $M$ for some $0 \leq r \leq n$, by Lemma~\ref{Lem-no0row}, we can consider the $(m-1) \times (n+1)$ matrix $N$ in the statement of the lemma instead of $M$, and every row of $N$ has a $0$.

\begin{theorem}\label{MainTheorem}
	Let $M$ be an $m \times n$ matrix over $\{-1,0,1\}$ such that every row of $M$ has a $0$. The directed split graph $G_o(M)$ is semi-transitive if and only if $M$ satisfies the following conditions:
	\begin{itemize}
		\item[{\em (i)}]  every row of $M$ is of the form $ 0^r  1^s  0^t $ or $ 0^r  (-1)^s  0^t $ or $ 1^r  0^s (-1)^t $ for $r, s, t\geq 0$, and

		\item[{\em (ii)}]  for each row of $M$ of the form $ 1^a 0^b  (-1)^c $ where $a,b,c>0$,  there is  no other row distinct from $ 1^a 0^b  (-1)^c $ such that its entries in positions $a$ and $a+b+1$ are not $0$'s.
	\end{itemize}
\end{theorem}

\begin{proof} ``$\Leftarrow$'' Note that the vertices in the independent set will then be of types A, B and C, and taking into account condition (ii),  Theorem~\ref{main-orientation} can be applied to see that $G_o(M)$ is semi-transitive.\\
	
\noindent
``$\Rightarrow$'' One can see that $G_o(M)$ is well-labelled, so the clique is oriented transitively and its longest path is  $1\rightarrow 2\rightarrow\cdots\rightarrow n$. Moreover, conditions (ii) and (iii) in Theorem~\ref{main-orientation} give conditions (i) and (ii) in this theorem.
\end{proof}

\begin{corollary}\label{CorMainTheorem<=}
Let $M$ be an $m \times n$ matrix over $\{-1,0,1\}$. The directed split graph $G_o(M)$ is semi-transitive if $M$ satisfies the following conditions:
	\begin{itemize}
		\item[{\em (i)}]  every row of $M$ is of the form $ 0^r  1^s  0^t $ or $ 0^r  (-1)^s  0^t $ or $ 1^r  0^s (-1)^t $ for $r, s, t\geq 0$, and
		
		\item[{\em (ii)}]  for each row of $M$ of the form $ 1^a 0^b  (-1)^c $ where $a,c>0$ and $b\geq 0$, there is  no other row distinct from $ 1^a 0^b  (-1)^c $ such that its entries in positions $a$ and $a+b+1$ are not $0$'s.
	\end{itemize}
\end{corollary}

\begin{proof}
	The intersection of row $i$ and column $j$ in $S_o(M)$ represents connection between vertices $i$ and $j$ in $G_o(M)$. Assume that $M$ satisfies the conditions (i) and (ii). If every row of $M$ has a $0$, then the result follows from Theorem~\ref{MainTheorem}. Suppose that there is a row $p$ of $M$ of the form $1^r(-1)^{n-r}$ where $1 \leq p \leq m$ and $0 \leq r \leq n$. Then, $\{1,2, \ldots,n, n+p\}$  is the maximal clique in the directed graph $G_o(M)$. By Lemma~\ref{Lem-no0row}, we have that $G_o(M)$ is isomorphic to $G_o(N)$, where $N$ is the matrix obtained from $M$ by deleting row $p$ and adding a zero-column between columns $r$ and $r+1$ (in the cases of $r=0$ and $r=n$, the zero-column will be the first column and the last column, respectively). Note that $N$ still satisfies conditions (i) and (ii) and every row of $N$ has a $0$. Applying Theorem~\ref{MainTheorem}, we have that $G_o(N)$ and $G_o(M)$ are semi-transitive.
\end{proof}

Corollary~\ref{CorMainTheorem<=} shows that the converse in Theorem~\ref{MainTheorem} is still true without the condition ``every row of $M$ has a $0$". However, the forward direction is not necessarily true then. Indeed, consider matrix $M=\begin{bmatrix} 1&0&-1&-1\\1&1&-1&-1\\0&0&1&0 \end{bmatrix}$. Note that $M$ does not satisfy condition (ii) in Theorem~\ref{MainTheorem}, while $G_o(M)$ is semi-transitive. So, the condition ``every row of $M$ has a $0$" is necessary for the forward direction in Theorem~\ref{MainTheorem}.

\begin{corollary}\label{CorMainTheorem=>}
	Let $M$ be an $m \times n$ matrix over $\{-1,0,1\}$. If the split graph $G_o(M)$ is semi-transitive, then every row of $M$ is of the form $ 0^r  1^s  0^t $ or $ 0^r  (-1)^s  0^t $ or $ 1^r  0^s (-1)^t $ for $r, s, t\geq 0$.
\end{corollary}

\begin{proof}
	Assume that $G_o(M)$ is semi-transitive. 
	If every row of $M$ has a $0$, then the result follows from Theorem~\ref{MainTheorem}. Suppose that there is a row $p$ in $M$ such that $m_{p1}m_{p2} \cdots m_{pn}$ is $1^{\ell}(-1)^{n-\ell}$ for some $0 \leq \ell \leq n$. So, the clique is oriented transitively and its longest path is $1\rightarrow 2\rightarrow\cdots \rightarrow\ell \rightarrow (n+p) \rightarrow (\ell+1) \rightarrow \cdots \rightarrow n$. Let $N$ be the matrix 
	$$\begin{bmatrix}
	m_{11}	  &m_{12}	 &\cdots&m_{1\ell}	   &0	  &m_{1(\ell+1)}&\cdots&m_{1n}\\	
	m_{21}	  &m_{22}	 &\cdots&m_{2\ell}	   &0	  &m_{2(\ell+1)}&\cdots&m_{2n}\\
	\vdots 	  &\vdots	 &      &\vdots    &\vdots&\vdots 	 &      &\vdots \\
	m_{(p-1)1}&m_{(p-1)2}&\cdots&m_{(p-1)\ell}&0	  &m_{(p-1)(\ell+1)}&\cdots&m_{(p-1)n}\\
	m_{(p+1)1}&m_{(p+1)2}&\cdots&m_{(p+1)\ell}&0	  &m_{(p+1)(\ell+1)}&\cdots&m_{(p+1)n}\\
	\vdots 	  &\vdots	 &      &\vdots    &\vdots&\vdots    &      & \vdots \\
	m_{m1}	  &m_{m2}	 &\cdots&m_{m\ell}	   &0	  &m_{m(\ell+1)}&\cdots&m_{mn}
	\end{bmatrix}.$$
	By Lemma~\ref{Lem-no0row}, we have that $G_o(M)$ is isomorphic to $G_o(N)$, and hence $G_o(N)$ is semi-transitive. Further, by Theorem~\ref{MainTheorem}, we have that $m_{i1}m_{i2} \cdots m_{i\ell} 0 m_{i(\ell+1)} \cdots m_{in}$ is of the form $ 0^r  1^s  0^t $ or $ 0^r  (-1)^s  0^t $ or $ 1^r  0^s (-1)^t $ for any $i \in \{1,2,\ldots p-1,p+1,\ldots,m\}$. Note that $m_{i1}m_{i2} \cdots m_{in}$ is also of the form $ 0^r  1^s  0^t $ or $ 0^r  (-1)^s  0^t $ or $ 1^r  0^s (-1)^t $ for any $i \in \{1,2,\ldots p-1,p+1,\ldots,m\}$. Hence, every row of $M$ is of the form $ 0^r  1^s  0^t $ or $ 0^r  (-1)^s  0^t $ or $ 1^r  0^s (-1)^t $ for $r, s, t\geq 0$.
\end{proof}

\begin{corollary}\label{corr-MainTheorem}
	Let $M$ be an $m \times n$ matrix over $\{-1,0,1\}$. If every row of $M$ is of the form $ 0^r  1^s  0^t $ or $ 0^r  (-1)^s  0^t $ for  $r, s, t\geq 0$, then the graph $G_o(M)$ is semi-transitive.
\end{corollary}
\begin{proof}
	If every row of $M$ has a $0$, then by Theorem~\ref{MainTheorem}, $G_o(M)$ is semi-transitive. Suppose there is an all $1$'s or all $(-1)$'s row in $M$. By Lemma~\ref{remark-no0not1(-1)}, we obtain a matrix $N$ such that every row has a $0$ and $G_o(M)$ is isomorphic to $G_o(N)$. Note that every row of $N$ is also of the form $ 0^x  1^y  0^z $ or $ 0^x  (-1)^y  0^z $ for $x, y, z\geq 0$. Thus, both $G_o(N)$ and $G_o(M)$ are semi-transitive.
\end{proof}

\begin{corollary}\label{010 and 10-1}
	If $M$ is a matrix over $\{-1,0,1\}$ containing a row of the form $11 \cdots 1$ and a row of the form $1^r0^s(-1)^t$ for some non-negative integers $r,s,t$ such that $r,t \neq 0$, then $G_o(M)$ is not semi-transitive.
\end{corollary}
\begin{proof}
	By Lemma~\ref{Lem-no0row}, the directed graph $G_o(M)$ is isomorphic to the graph $G_o(N)$, where $N$ is a matrix such that every row of $N$ has a $0$ and $N$ contains a row of the form $1^r0^s(-1)^t0$. By Corollary~\ref{CorMainTheorem=>}, $G_o(N)$ is not semi-transitive, and so is $G_o(M)$.
\end{proof}

\begin{corollary} \label{0-10 and 10-1}
 		If $M$ is a matrix over $\{-1,0,1\}$ containing a row of the form $(-1)(-1) \cdots (-1)$ and a row of the form $1^r0^s(-1)^t$ for some non-negative integers $r,s,t$ such that $r,t \neq 0$, then $G_o(M)$ is not semi-transitive.
	\end{corollary}
\begin{proof}
	By Lemma~\ref{Lem-no0row}, the directed graph $G_o(M)$ is isomorphic to the graph $G_o(N)$, where $N$ is a matrix without all $1$'s or all $(-1)$'s row and $N$ contains a row of the form $01^r0^s(-1)^t$. By Corollary~\ref{CorMainTheorem=>}, $G_o(N)$ is not semi-transitive, and so is $G_o(M)$.
\end{proof}

\begin{definition}
	A matrix $M$ is said to be a {\em layered matrix} if all entries in the same row of $M$ are identical.
\end{definition}

The next result is a straightforward corollary of Corollary~\ref{corr-MainTheorem}.

\begin{corollary}\label{Mall0all1}
	Let $M$ be an $m\times n$ matrix over $\{-1,0,1\}$. If $M$ is a layered matrix, then $G_o(M)$ is semi-transitive.
\end{corollary}

\section{Directed split graphs generated by iterations of morphisms}

\begin{definition}
	Let  $A,B,C$ be $m\times n$ matrices over $\{-1,0,1\}$. The matrix $M^k(A,B,C)$ is the $k^{th}$-iteration of the $2$-dimensional morphism applied to the $1 \times 1$ matrix $\left[ 0 \right]$ which maps $\left[ 0 \right] \rightarrow A$, $\left[ 1 \right] \rightarrow B$ and $\left[ -1 \right] \rightarrow C$. Moreover, we write $S_o^k(A,B,C)$ for the matrix $S_o(M^k(A,B,C))$ and $G_o^k(A,B,C)$ for the graph with the adjacency matrix $S_o^k(A,B,C)$.
\end{definition}

	\begin{example} \label{ExABC}
		Let $A=\begin{bmatrix}0&1\\0&-1\end{bmatrix}$, $B=\begin{bmatrix}-1&-1\\1&0	\end{bmatrix}$  and $C=\begin{bmatrix}1&1\\-1&-1 \end{bmatrix}$. Then we have	$M^0(A,B,C)=\begin{bmatrix}0\end{bmatrix}$, 
		$M^1(A,B,C)=\begin{bmatrix}0&1\\0&-1\end{bmatrix}$ and\\
		$M^2(A,B,C)=\begin{bmatrix}0&1&-1&-1\\0&-1&1&0\\0&1&1&1\\0&-1&-1&-1\end{bmatrix}$.
		Hence, $S_o^2(A,B,C)$ is the matrix $$\begin{bmatrix} 0&0&0&0&0&0&0&0\\
		1&0&0&0&0&0&0&0\\
		1&1&0&0&0&0&0&0\\
		1&1&1&0&0&0&0&0\\
		0&1&-1&-1&0&0&0&0\\
		0&-1&1&0&0&0&0&0\\
		0&1&1&1&0&0&0&0\\
		0&-1&-1&-1&0&0&0&0\end{bmatrix} $$ and $G_o^2(A,B,C)$ is shown in Figure~\ref{ExampleSABC}.
	\end{example}
	
	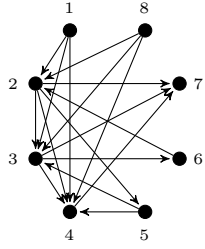
\begin{figure}
		\begin{center}	
				\begin{tikzpicture}[->,>=stealth', shorten >=1pt,node distance=1cm,auto,main node/.style={fill,circle,draw,inner sep=0pt,minimum size=5pt}]

				\node[main node] (1) {};	
				\node [above of=1, xshift=-0.1mm, yshift=-7mm] {\tiny{1}};
				\node[main node] (2) [below left of=1, xshift=0.25cm] {};
				\node [left of=2, xshift=7mm, yshift=0] {\tiny{2}};
				\node[main node] (3) [below of=2] {};
				\node [left of=3, xshift=7mm, yshift=0] {\tiny{3}};
				\node[main node] (4) [below right of=3, xshift=-0.25cm] {};
				\node [below of=4, xshift=-0.1mm, yshift=7mm] {\tiny{4}};

				\node[main node] (5) [right of=4] {};
				\node [below of=5, xshift=-0.1mm, yshift=7mm] {\tiny{5}};
                                  \node[main node] (6) [above right of=5, xshift=-0.25cm] {};
				\node [right of=6, xshift=-7.5mm, yshift=0] {\tiny{6}};
				\node[main node] (7) [above of=6] {};
				\node [right of=7, xshift=-7.5mm, yshift=0] {\tiny{7}};
				\node[main node] (8) [above left of=7,xshift=0.25cm] {};
				\node [above of=8, xshift=-0.1mm, yshift=-7mm] {\tiny{8}};
				
				\path
				(1) edge (2)
				(1) edge (3)
				(1) edge (4)
				(2) edge (3)
				(2) edge (4)
				(3) edge (4)
				(2) edge (5)
				(5) edge (3)
				(5) edge (4)
				(6) edge (2)
				(3) edge (6)
				(2) edge (7)
				(3) edge (7)
				(4) edge (7)
				(8) edge (2)
				(8) edge (3)
				(8) edge (4);
				\end{tikzpicture}

		\caption{\label{ExampleSABC}  The directed split graph $G_o^2(A,B,C)$ corresponding to the adjacency matrix $S_o^2(A,B,C)$ in Example \ref{ExABC}.}
		\end{center}
	\end{figure}

\begin{remark}\label{A=0}
	If $A$ is a zero matrix, then $M^k(A,B,C)$ is always a zero matrix for any $m \times n$ matrices $B$ and $C$ and  $k\geq0$. Thus, by Remark~\ref{M=0}, $G_o^k(A,B,C)$ is semi-transitive in this case. 
\end{remark}

\begin{proposition}\label{ABCallLayered}
	If $A,B$ and $C$ are layered matrices over $\{-1,0,1\}$, then $G_o^k(A,B,C)$ is semi-transitive for any $k\geq 0$. 
\end{proposition}

\begin{proof} Let $A$, $B$ and $C$ be $m \times n$ matrices. Since every row in $A, B$ and $C$ is either $0^n$ or  $1^n$ or $(-1)^n$, we have that every row in $M^k(A,B,C)$ is either $0^{n^k}$ or  $1^{n^k}$ or $(-1)^{n^k}$, so by Corollary~\ref{Mall0all1}, $G_o^k(A,B,C)$ is semi-transitive. \end{proof}

If $A=[a_{ij}]_{m \times n}$ contains at least one $0$, say $a_{ij}=0$, then the entry in row $i$ and column $j$ of $M^1(A,B,C)$ is 0. By mapping this $0$ to $A$ in the next iteration of morphism, we obtain $A=M^1(A,B,C)$ as the $m \times n$ submatrix of $M^2(A,B,C)$ given by intersection of rows $(i-1)n+1,(i-1)n+2,\ldots,in$ and columns $(j-1)m+1,(j-1)m+2,\ldots,jm$. More generally, the $m^{k-1} \times n^{k-1}$ submatrix of $M^k(A,B,C)$ given by intersection of rows $(i-1)n^{k-1}+1,(i-1)n^{k-1}+2,\ldots,in^{k-1}$ and columns $(j-1)m^{k-1}+1,(j-1)m^{k-1}+2,\ldots,jm^{k-1}$ is $M^{k-1}(A,B,C)$. So, we can consider the bottommost, then leftmost zero in $A$ as the start of a chain of induced subgraphs generated by the morphism. Thus, the limit $\lim_{k \to \infty}M^k(A,B,C)$, called a {\em fixed point of the morphism}, is well-defined. So, we have that $G_o^i(A,B,C)$ is an induced subgraph of $G_o^k(A,B,C)$ for $i \leq k$, and the notion of the infinite split graph $G_o(A,B,C)$ is well-defined in the case when $A$ has a 0. Note that this is not a necessary condition for $G_o(A,B,C)$ to be well-defined (for example, $A, B, C$ could be all one matrices). We are interested in the smallest integer $\ell$ (possibly non-existing) such that $G_o^{\ell}(A,B,C)$ is not semi-transitive for given $A, B$ and $C$ (then $G_o^i(A,B)$ is not semi-transitive for $i\geq \ell$).

\begin{definition}\label{IWR-def} Let $A,B,C$ be $m \times n$ matrices such that $A$ has a $0$ as an entry. The {\em index of semi-transitivity} $\IST(A,B,C)$ of an infinite directed split  graph $G_o(A,B,C)$ is the smallest integer $\ell$ such that $G_o^{\ell}(A,B,C)$ is not semi-transitive. If such an $\ell$ does not exist, that is, if $G_o^{\ell}(A,B,C)$ is semi-transitive for all $\ell$, then $\ell:=\infty$. 
\end{definition}

Note that since $G_o^0(A,B,C)$ is a graph with one vertex for any $A,B,C$, we have $\IST(A,B,C) \geq 1$.  

\begin{remark}\label{}
	It follows from Proposition~\ref{ABCallLayered} that  $\IST(A,B,C) = \infty$ if $A,B$ and $C$ are layered matrices. 
\end{remark}

The following three lemmas give sufficient conditions for $A, B$ and $C$ to have $\IST(A,B,C)= \infty$.

\begin{lemma} \label{Lem-1res}
		Let $A, B$ and $C$ be $m \times n$ matrices over $\{-1,0,1\}$ such that $A$ has a $0$ and $\IST(A,B,C)= \infty$. Then, 
	\begin{itemize}
		\item If $A$ is not a layered matrix, then there is no row in $M^k(A,B,C)$ containing two $0$'s for any $k\geq 0$.
		\item If $B$ is not a layered matrix, then there is no row in $M^k(A,B,C)$ containing two $1$'s for any $k\geq 0$.
		\item If $C$ is not a layered matrix, then there is no row in $M^k(A,B,C)$ containing two $(-1)$'s for any $k\geq 0$.
	\end{itemize}
\end{lemma}

\begin{proof}
We will prove the first bullet point; the other bullet points can be proved analogously. 

Let $A=[a_{ij}]$ be an $m \times n$ matrix and $a_{ir},a_{is}$ be two entries in row $i$ of $A$ such that $a_{ir} \neq a_{is}$ where $1\leq r<s \leq n$. Denote $\mu^k(i,j) \in \{-1,0,1\}$ the entry of $M^k(A,B,C)$ in row $i$ and column $j$. Suppose that  row $a$ of $M^k(A,B,C)$ contains at least two $0$'s for some $k$, say $\mu^k(a,b)=\mu^k(a,c)=0$ where $b<c$. Consider  the intersection of rows $(a-1)m+1,(a-1)m+2,\ldots,am$ and columns $(b-1)n+1,(b-1)n+2, \ldots, bn$ in $M^{k+1}(A,B,C)$, which is the matrix $A$ because $\mu^k(a,b) =0$. Similarly, the submatrix of $M^{k+1}(A,B,C)$ formed by rows $(a-1)m+1,(a-1)m+2,\ldots,am$ and columns $(c-1)n+1,(c-1)n+2, \ldots, cn$ is $A$. Hence, we have
	$$\mu^{k+1}((a-1)m+i,(b-1)n+r) = \mu^{k+1}((a-1)m+i,(c-1)n+r) = a_{ir}$$ and $$\mu^{k+1}((a-1)mi,(b-1)n+s) = \mu^{k+1}((a-1)m+i,(c-1)n+s) = a_{is}.$$ 
Thus, the submatrix of $M^{k+1}(A,B,C)$ formed by row $(a-1)m+i$ and columns $(b-1)n+r,(b-1)n+s,(c-1)n+r,(c-1)n+s$ is $\begin{bmatrix}
a_{ir},a_{is},a_{ir},a_{is}
\end{bmatrix}$. That is, row $(a-1)m+i$ of $M^{k+1}(A,B,C)$ cannot be of the form $ 0^r  1^s  0^t $ or $ 0^r  (-1)^s  0^t $ or $ 1^r  0^s (-1)^t $. By Corollary~\ref{CorMainTheorem=>}, $G_o^{k+1}(A,B,C)$ is not semi-transitive, which is a contradiction with $\IST(A,B,C) = \infty$.
\end{proof} 

\begin{lemma}\label{Lem-2res}
	Let $A, B$ and $C$ be $m \times n$  matrices over $\{-1,0,1\}$ such that $A$ has a $0$ and $\IST(A,B,C)= \infty$. Then, 
	\begin{itemize}
		\item If $A$ and $B$ are not layered matrices, then every entry of $C$ is $(-1)$.
		\item If $A$ and $C$ are not layered matrices, then every entry of $B$ is $1$.
	\end{itemize}
\end{lemma}
\begin{proof}
	Both statements are proved by similar arguments, so we will prove here only the first one. Suppose both $A$ and $B$ are not layered matrices. By Lemma~\ref{Lem-1res}, every row of $M^k(A,B,C)$ contains at most one $0$ and at most one $1$ for any $k \geq 2$. Then, there are at least $n^k-2$ copies of $(-1)$ in every row of $M^k(A,B,C)$. By Lemma~\ref{Lem-1res}, $C$ is a layered matrix.

	Suppose that there is no $(-1)$ in $A$ and $B$. Since every row of $M^1(A,B,C) = A$ has at most one $0$ and at most one $1$ and no $(-1)$, then $n=2$ (recall our assumption of $n\geq 2$). Therefore, $M^2(A,B,C)$ has $4$ columns with every row having more than one $0$ or more than one $1$, which is a contradiction.
	
	If $(-1)$ is an entry of $A$, then $M^1(A,B,C)=A$ has $(-1)$ as an entry. So $C$ is a submatrix of $M^2(A,B,C)$ as $(-1)$ is mapped to $C$. Since every row of $C$ has the same entries, and there is no more than one $0$ and one $1$ in each row of $M^2(A,B,C)$, we have that each entry of $C$ must be $(-1)$.
	
	Finally, if there is no $(-1)$ in $A$, but $B$ contains $(-1)$ as an entry, then $M^1(A,B,C)=A$ contains $1$ as an entry. Since $1$ maps to $B$, $M^2(A,B,C)$ contains $B$ as a submatrix. So there is an entry $(-1)$ in $M^2(A,B,C)$, and then $C$ is a submatrix of $M^3(A,B,C)$. Since every row of $C$ has entries equal to each other, and there is no more than one $0$ and one $1$ in each row of $M^2(A,B,C)$, then each entry of $C$ is $(-1)$.
\end{proof}

\begin{lemma}\label{Lem-2resA}
	Let $A, B$ and $C$ be $m \times n$  matrices over $\{-1,0,1\}$ such that $A$ has a $0$ and $\IST(A,B,C)= \infty$. If $B$ and $C$ are not layered matrices, then all entries of $A$ are $0$.
\end{lemma}
\begin{proof}
	Suppose $B$ and $C$ are not layered matrices. By Lemma~\ref{Lem-1res}, every row of $M^k(A,B,C)$ contains at most one $1$ and at most one $(-1)$ for any $k \geq 2$. Then there are at least $n^k-2$ zeroes in every row of $M^k(A,B,C)$. By Lemma~\ref{Lem-1res}, $A$ is a layered matrix.
	
	Assume that there is a row $r$ in $A:=[a_{ij}]=M^1(A,B,C)$ of the form $11 \cdots 1$. Also, suppose that a row $s$ in $B:=[b_{ij}]$ has two distinct entries, say $b_{sp} \neq b_{sq}$ for some $1 \leq p<q \leq n$. Note that the intersection of rows $(r-1)m+1, (r-1)m+2, \ldots, rm$ and columns $(\ell-1)n+1, (\ell-1)n+2, \ldots, \ell n$ in $M^2(A,B,C)$ is $B$ for $\ell=1,2,\ldots,m$. Then the submatrix of $M^2(A,B,C)$ formed by row $(r-1)m+s$ and columns $p, q, n+p, n+q, 2n+p, 2n+q, \ldots, (m-1)n+p, (m-1)n+q$ is $$\begin{bmatrix}
	b_{sp} & b_{sq} & b_{sp} &b_{sq}  & \cdots& b_{sp} & b_{sq}
	\end{bmatrix}.$$
	Since every row of $M^k(A,B,C)$ has at most one $1$ and at most one $(-1)$ for any $k$, we have $b_{sp} = b_{sq} = 0$, which is a contradiction. Thus, there is no row in $A$ of the form $11 \cdots 1$. Similarly, we can show that there is no row in $A$ of the form $(-1)(-1) \cdots (-1)$. Hence, $A$ is an all $0$ matrix.
\end{proof}

From Lemmas~\ref{Lem-2res} and~\ref{Lem-2resA} we have the following theorem.
\begin{theorem}\label{Thm-ABC-not-layered}
	Let $A, B$ and $C$ be $m \times n$  matrices over $\{-1,0,1\}$ such that $A$ has a $0$. If $A$, $B$ and $C$ are not layered, then $\IST(A,B,C)$ is finite. 
\end{theorem}

\begin{definition} \label{DefIndependent}
	Let $A,B,C$ be $m \times n$ matrices over $\{-1,0,1\}$. The triple $(A,B,C)$ is said to be independent from $B$ if there are no $1$'s in $A$ and $C$. Similarly, the triple $(A,B,C)$ is said to be independent from $C$ if there are no $(-1)$'s in $A$ and $B$. 
\end{definition}

For convenience, we write $R(M)$ for the set of strings representing rows of $M$.  Moreover, if $A, B$ and $C$ are $m \times n$ matrices over $\{-1,0,1\}$, then define $R^k(A,B,C)$ to be the set of strings representing rows of $M^k(A,B,C)$. So, every element of $R^k(A,B,C)$ is a string over $\{-1,0,1\}$ of length $n^k$. Each element of $R^k(A,B,C)$ is called a {\em row pattern of $M^k(A,B,C)$}.

\begin{theorem} \label{thm-M(A,B,C)IndependentFromC}
	Let $A,B$ and $C$ be $m \times n$ matrices over $\{-1,0,1\}$ such that $A$ has a $0$ and $(A,B,C)$ is independent from $C$. Then, $\IST(A,B,C)= \infty$ if and only if $A$ and $B$ satisfy one of the following conditions, where $a_i \in \{0,1\}$:
	\begin{itemize}
		\item[$(1)$] $A$ and $B$ are layered matrices, or
		\item[$(2)$] $A= \begin{bmatrix}
		a_1    & 1      & 1      & \cdots & 1 \\
		a_2    & 1      & 1      & \cdots & 1 \\ 
		\vdots & \vdots & \vdots &  	  & \vdots \\
		a_m    & 1      & 1      & \cdots &1 \end{bmatrix}$ and 
		$B= \begin{bmatrix}
		1      &1       & \cdots & 1  \\
		1      &1       & \cdots & 1  \\
		\vdots & \vdots &        & \vdots  \\
		1      &1       & \cdots & 1  \end{bmatrix}$, or
		
		\item[$(3)$] $A= \begin{bmatrix}
		 1      & 1      & \cdots & 1      & a_1   \\
		 1      & 1      & \cdots & 1      & a_2   \\ 
		\vdots  & \vdots &        & \vdots & \vdots \\
		 1      & 1      & \cdots & 1      &a_m    \end{bmatrix}$ and 
		 $B= \begin{bmatrix}1      &1       & \cdots & 1  \\
		 1      &1       & \cdots & 1  \\
		 \vdots & \vdots &        & \vdots  \\
		 1      &1       & \cdots & 1 \end{bmatrix}$.
	\end{itemize}
\end{theorem}
\begin{proof}
	``$\Leftarrow$'' There is no $(-1)$ in $A$ and $B$, and row patterns of $M^k(A,B,C)$ generated by $A, B$ and $C$ in (1), (2) and (3) are in the set $$\{1^{n^k}, 0^{n^k}, 01^{n^k-1}, 1^{n^k-1}0\}.$$ By Corollary~\ref{corr-MainTheorem}, $M^k(A,B,C)$ is semi-transitive for all $k\geq 0$. \\[-2mm]

\noindent
	``$\Rightarrow$''  Since $(A,B,C)$ is independent from $C$, every entry of $M^k(A,B,C)$ is either 0 or 1. Assume $\IST(A,B,C)=\infty$ and let $R(B)=\{b_1,b_2, \ldots, b_p\}$ where $b_i$ is a binary string of length $n$. By Corollary~\ref{CorMainTheorem=>}, we have that every row of $M^k(A,B,C)$ is of the form $0^r1^s0^t$. If $A$ is a layered matrix, then $R^1(A,B,C)\subseteq\{0^n,1^n\}$ and $$R^2(A,B,C)\subseteq\{ 0^{n^2}, 1^{n^2}, (b_1)^{n}, (b_2)^{n},\ldots , (b_p)^{n} \}.$$ So, $R(B) \subseteq \{0^n,1^n\}$ as otherwise, some strings in $R^2(A,B,C)$ are not of the form $0^r1^s0^t$. Thus, $B$ is a layered matrix.
	Suppose $A$ is not a layered matrix. By Lemma~\ref{Lem-1res}, $R^1(A,B,C) \subseteq \{01^{n-1}, 1^{n-1}0, 1^n\}$ . If both $01^{n-1}$ and $1^{n-1}0$ are rows in $A$, then $1^{n-1}0(b_i)^{n-1}$ is a row pattern in $R^2(A,B,C)$ for some $i$. Since every row of $M^k(A,B,C)$ contains at most one $0$, $b_i$ must be $1^n$, which contradicts $1^{n-1}0(b_i)^{n-1}$ not being of the form $0^r1^s0^t$. So, we have 
	$$A= \begin{bmatrix}
	a_1    & 1      & 1      & \cdots & 1 \\
	a_2    & 1      & 1      & \cdots & 1 \\ 
	\vdots & \vdots & \vdots &  	  & \vdots \\
	a_m    & 1      & 1      & \cdots &1 \end{bmatrix} \text{~~or~~}
	A= \begin{bmatrix}
	1      & 1      & \cdots & 1      & a_1   \\
	1      & 1      & \cdots & 1      & a_2   \\ 
	\vdots  & \vdots &  	 & \vdots & \vdots \\
	1      & 1      & \cdots & 1      &a_m    \end{bmatrix}$$
	where $a_i \in \{0,1\}$. Note that each row of $A$ is $1^n$, $01^{n-1}$ or $1^{n-1}0$. If row $i$ in $A$ is $1^n$, then row $((i-1)m+i)$ in $M^2(A,B,C)$ is $x^n$, where $x$ is row $i$ in $B$. Since $x^n$ cannot contain more than one $0$, we have $x= 1^n$. If row $i$ in $A$ is $01^{n-1}$, then row $((i-1)m+i)$ in $M^2(A,B,C)$ is $01^{n-1}x^{n-1}$, where $x$ is row $i$ in $B$. So, $x= 1^n$ because $01^{n-1}x^{n-1}$ contains at most one $0$. Similarly, if row $i$ in $A$ is $1^{n-1}0$, then row $i$ in $B$ is $1^n$. Hence,  $B$ is an all 1 matrix. \end{proof}

Next theorem can be proved similarly to Theorem~\ref{thm-M(A,B,C)IndependentFromC}.

\begin{theorem} \label{thm-M(A,B,C)IndependentFromB}
	Let $A,B$ and $C$ be $m \times n$ matrices over $\{-1,0,1\}$ such that $A$ has a $0$ and  $(A,B,C)$ is independent from $B$. Then, $\IST(A,B,C)= \infty$ if and only if $A$ and $C$ satisfy one of the following conditions, where $a_i \in \{0,1\}$:
	\begin{itemize}
		\item[$(1)$] $A$ and $C$ are layered matrices, or
		\item[$(2)$] $A= \begin{bmatrix}
		a_1    & -1     & -1     & \cdots & -1 \\
		a_2    & -1     & -1     & \cdots & -1 \\ 
		\vdots & \vdots & \vdots &  	  & \vdots \\
		a_m    & -1     & -1     & \cdots & -1 \end{bmatrix}$ and 
		$C= \begin{bmatrix}
		-1     &-1      & \cdots & -1  \\
		-1     &-1      & \cdots & -1  \\
		\vdots & \vdots &        & \vdots  \\
		-1     &-1      & \cdots & -1  \end{bmatrix}$, or
		
		\item[$(3)$] $A= \begin{bmatrix}
		-1     & -1     & \cdots & -1     & a_1   \\
		-1     & -1     & \cdots & -1     & a_2   \\ 
		\vdots & \vdots &        & \vdots & \vdots \\
		-1     & -1     & \cdots & -1     &a_m    \end{bmatrix}$ and 
		$C= \begin{bmatrix}
		-1     &-1      & \cdots & -1  \\
		-1     &-1      & \cdots & -1  \\
		\vdots & \vdots &        & \vdots  \\
		-1     &-1      & \cdots & -1 \end{bmatrix}$.
	\end{itemize}
\end{theorem}

\begin{theorem} \label{thm-notIndeoendent-A=L}
	Let $A$, $B$ and $C$ be $m \times n$ matrices over $\{-1,0,1\}$ such that $A$ has a $0$ and  $(A,B,C)$ is not independent from $B$ and $C$. Suppose $A$ is a layered matrix. Then, $\IST(A,B,C)= \infty$ if and only if $B$ and $C$ are layered matrices.
\end{theorem}
\begin{proof}
	Suppose $\IST(A,B,C)= \infty$. The case when $A$ is a zero matrix is trivial. Thus, assume that $1^n$ or $(-1)^n$ is a row in $A$. W.L.O.G., we suppose that $1^n$ is a row in $A=M^1(A,B,C)$. By Lemma~\ref{Lem-1res}, we have $B$ is a layered matrix. If $A$ also contains a row $(-1)^n$, then $C$ is a layered matrix with the same reason. If $A$ does not contain a row $(-1)^n$, then $(-1)^n$ must be a row of $B$ because $(A,B,C)$ is not independent from $B$ and $C$. Since  $1^n$ is a row of $A$, we have $BB \cdots B$ are $m$ consecutive rows in $M^2(A,B,C)$. As $(-1)^n$ is a row in $B$, we have that $(-1)^{n^2}$ is a row in $M^2(A,B,C)$. By Lemma~\ref{Lem-1res}, $C$ is a layered matrix.
	
	For the converse direction, it is clear from Proposition~\ref{ABCallLayered}  that if $A$, $B$ and $C$ are layered matrices, then $\IST(A,B,C) = \infty$.
\end{proof}

\begin{definition}\label{all-but-triples-DEF}
	Let $A$, $B$, $C$ be $m \times n$ matrices over $\{-1,0,1\}$. The triple $(A,B,C)$ is said to be
	\begin{itemize}
		\item an {\em all-but-leftmost-negative triple} if $R(A), R(B) \subseteq \{0(-1)^{n-1}, 1(-1)^{n-1}\}$ and $C$ is an all $(-1)$ matrix,
		\item an {\em all-but-rightmost-negative triple} if $R(A), R(B) \subseteq \{(-1)^{n-1}0, (-1)^{n-1}1\}$ and $C$ is an all $(-1)$ matrix,
		\item an {\em all-but-leftmost-positive triple} if $R(A), R(B) \subseteq\{01^{n-1}, (-1)1^{n-1}\}$ and $C$ is an all $1$ matrix,
		\item an {\em all-but-rightmost-positive triple} if $R(A), R(B) \subseteq \{1^{n-1}0,$ $ 1^{n-1}(-1)\}$ and $C$ is an all $1$ matrix.
	\end{itemize}
\end{definition}

From Definition~\ref{all-but-triples-DEF}, we can easily see that 
\begin{itemize}
	\item If $(A,B,C)$ is all-but-leftmost-negative, then\\ \centerline{$R^k(A,B,C) \subseteq \{0(-1)^{n^k-1}, 1(-1)^{n^k-1}\}$,}
	\item If $(A,B,C)$ is all-but-rightmost-negative, then\\ \centerline{$R^k(A,B,C) \subseteq \{(-1)^{n^k-1}0, (-1)^{n^k-1}1\}$,}
	\item If $(A,B,C)$ is all-but-leftmost-positive, then\\ \centerline{$R^k(A,B,C) \subseteq \{01^{n^k-1}, (-1)1^{n^k-1}\}$,}
	\item If $(A,B,C)$ is all-but-rightmost-positive, then\\ \centerline{$R^k(A,B,C) \subseteq \{1^{n^k-1}0, 1^{n^k-1}(-1)\}$.}
\end{itemize}
With this observation, we can prove the following theorem.

\begin{theorem} \label{thm-notIndeoendent-A,B=NL-C=L}
	Let $A$, $B$, $C$ be $m \times n$ matrices over $\{-1,0,1\}$ such that $A$ has a $0$ and $(A,B,C)$ is not independent from $B$ and $C$. Suppose $A$ and $B$ are not layered matrices and $C$ is a layered matrix. Then, $\IST(A,B,C) = \infty$ if and only if $(A,B,C)$ is an  all-but-leftmost-negative triple.
\end{theorem}
\begin{proof}
	``$\Leftarrow$'' Let $(A,B,C)$ be all-but-leftmost-negative. Then, for any $k \geq1$, $$ M^k(A,B,C)=\begin{bmatrix}
	x_1 	& -1 	& -1 	& \cdots 	& -1\\
	x_2 	& -1 	& -1	& \cdots 	& -1\\
	\vdots 	&\vdots &\vdots	&  			& \vdots \\
	x_{m^k} & -1 	& -1	& \cdots 	& -1
	\end{bmatrix}$$ 
	where $x_i \in \{0,1\}$. So $M^k(A,B,C)$ satisfies both conditions in Corollary~\ref{CorMainTheorem<=}, and hence  $\IST(A,B,C) = \infty$.\\[-2mm]
	
\noindent
``$\Rightarrow$'' Suppose $\IST(A,B,C)= \infty$. From Lemma~\ref{Lem-2res}, we have that $C$ is an all $(-1)$ matrix. By Lemma~\ref{Lem-1res}, every row of $M^k(A,B,C)$ does not contain more than one $0$ and more than one $1$. Note that every row of $A$ must be of the form $0^r1^s0^t$, $0^r(-1)^s0^t$ or $1^r0^s(-1)^t$, where $r,s,t\geq 0$. So, all possible row patterns of $A$ are in 

\ \ \ \ \  $\{01, 10, 0(-1)^{n-1}, (-1)^{n-1}0, (-1)^n, 1(-1)^{n-1},10(-1)^{n-2}\}$.
	
	Suppose that $n=2$ and row $i$ in $A$ is $01$. Then, the submatrix of $M^2(A,B,C)$ formed by rows $(i-1)m+1, (i-1)m+2, \ldots, im$ and columns $1, 2, 3, 4$ is $AB$. So, row $(i-1)m+i$ in $M^2(A,B,C)$ is $01x$, where $x$ is row $i$ in $B$. Note that $01x$ must be of the form $0^r1^s0^t$, where $r,s,t\geq 0$. Therefore, $x$ is $11$ because $M^2(A,B,C)$ contains at most one $0$. So, $01x$ contains more than one $1$, which contradicts Lemma~\ref{Lem-1res}. Hence, $01$ cannot be a row in $A$. Similarly, we obtain that $10$ is also not a row in $A$. Hence, we have that $01$ and $10$ cannot be a row in $A$.
	
 Suppose row $i$ in $A$ is $10(-1)^{n-2}$. Then there is $m$ consecutive rows in $M^2(A,B,C)$ built by $BACC\cdots C$. Note that row $i$ in $BACC \cdots C$ is $y10(-1)^{n-2}zz \cdots z$ , where $y$ and $z$ are rows $i$ in $B$ and $C$, respectively. Since $\IST(A,B,C) = \infty$, $y10(-1)^{n-2}zz \cdots z$ must be of the form $1^r0^s(-1)^t$, where $r,s,t\geq 0$. Thus, $y=1^n$ and $z=(-1)^{n}$. This contradicts to the fact that any row in $M^2(A,B,C)$ has at most one 1. Hence, $10(-1)^{n-2}$ cannot be a row in $A$.
	
	Now, all possible row patterns of $A$ are in 
	
\centerline{$\{0(-1)^{n-1}, (-1)^{n-1}0, (-1)^n,1(-1)^{n-1}\}$. }

\noindent
If $1(-1)^{n-1}$ is not a row in $A$, then $(A,B,C)$ is independent from $C$. Then $1(-1)^{n-1}$ must be a row in $A$. By Lemma~\ref{Lem-no0row}, we have $G^1_o(A,B,C)$ is isomorphic to $G_o(A^*)$ where $A^*$ is the matrix obtained by deleting row $1(-1)^{n-1}$ from $A$ and after that adding a zero column between the first and the second columns. If $(-1)^{n-1}0$ or $(-1)^{n-1}$ is a row in $A$, then $(-1)0(-1)^{n-3}0$ or $(-1)0(-1)^{n-2}$ is a row in $A^*$, respectively. By Corollary~\ref{CorMainTheorem=>}, $G_o(A^*)$ and  $G^1_o(A,B,C)$ are not semi-transitive. Therefore $(-1)^{n-1}0$ and $(-1)^{n}$ are not rows in $A$ and we have $$ A=\begin{bmatrix}
	a_1 	& -1 	& -1 	& \cdots 	& -1\\
	a_2 	& -1 	& -1	& \cdots 	& -1\\
	\vdots 	&\vdots &\vdots	& 			& \vdots \\
	a_m 	& -1 	& -1	& \cdots 	& -1
	\end{bmatrix} \text{~~where~~}a_i \in \{0,1\}.$$

	Since both $0(-1)^{n-1}$ and $1(-1)^{n-1}$ are rows in $A$, there are $m$ consecutive rows of $M^2(A,B,C)$ built by $ACC \cdots C$ and $BCC \cdots C$. Then $1(-1)^{n^2-1}$ is a row in $M^2(A,B,C)$.
	By Lemma~\ref{Lem-no0row}, we have $G_o^2(A,B,C)$ is isomorphic to $G_o(N)$ where $N$ is a matrix obtained by deleting row $1(-1)^{n^2-1}$ from $M^2(A,B,C)$ and after that adding a zero column between the first and the second columns. Note that a row $i$ of $BCC \cdots C$ is $b_{i1}b_{i2}\cdots b_{in}(-1)^{n^2-n}$ where $b_{i1}b_{i2}\cdots b_{in}$ is the row $i$ of $B$. Since $G_o(N)$ is semi-transitive and $b_{i1}0b_{i2}\cdots b_{in}(-1)^{n^2-n}$ is a row of $N$, we have $b_{i1}0b_{i2}\cdots b_{in}$ is $0^{r+1}(-1)^{n+1-r}$ or $10^r(-1)^{n-r}$ for some $1 \leq r \leq n$. As $M^2(A,B,C)$ contains at most one $0$, we obtain that $b_{i1}0b_{i2}\cdots b_{in}$ must be $10(-1)^{n-1}$ or $0^2(-1)^{n-1}$ for any $1 \leq i \leq m$. That is, $$ B=\begin{bmatrix}
		b_1 	& -1 	& -1 	& \cdots 	& -1\\
		b_2 	& -1 	& -1	& \cdots 	& -1\\
		\vdots 	&\vdots &\vdots	&			& \vdots \\
		b_m & -1 	& -1	& \cdots 	& -1
		\end{bmatrix} \text{~~where~~}b_i \in \{0,1\}.$$
	
\end{proof}
Using similar arguments, we can prove the following theorem.

\begin{theorem} \label{thm-notIndeoendent-A,C=NL-B=L}
	Let $A$, $B$, $C$ be $m \times n$ matrices over $\{-1,0,1\}$ such that $A$ has a $0$ and $(A,B,C)$ is not independent from $B$ and $C$. Suppose $A$ and $C$ are not layered matrices and $B$ is a layered matrix. Then, $\IST(A,B,C) = \infty$ if and only if $(A,B,C)$ is all-but-rightmost-positive.
\end{theorem}

By now, we already have classification for triples $(A,B,C)$ with the index of semi-transitivity infinity except for the case when $A$ is not a layered matrix and $B$ and $C$ are layered matrices and  $(A,B,C)$ is not independent from $B$, $C$. To solve the remaining cases, we need the following four lemmas.

\begin{lemma} \label{Lem-5properties}
	Let $A$, $B$, $C$ be $m \times n$ matrices over $\{-1,0,1\}$ such that $A$ has a $0$ and $(A,B,C)$ is not independent from $B$ and $C$. Then, 
	\begin{itemize}
		\item[$(1)$] \label{sublemop1} if $01^{n-1}$ and $1^{n-1}0$ are rows in $A$, then $\IST(A,B,C)$ is finite; 
		\item[$(2)$] \label{sublemop2} if $0(-1)^{n-1}$ and $(-1)^{n-1}0$ are rows in $A$, then $\IST(A,B,C)$ is finite; 
		\item[$(3)$] \label{sublemop3} if $01^{n-1}$ and $(-1)^{n-1}0$ are rows in $A$, then $\IST(A,B,C)$ is finite;	
		\item[$(4)$] \label{sublemop4} if $0(-1)^{n-1}$ and $1^{n-1}0$ are rows in $A$, then $\IST(A,B,C)$ is finite; 
		\item[$(5)$] \label{sublemop5} if $1^p0(-1)^{n-p-1}$ and $1^q0(-1)^{n-q-1}$ are rows in $A$, where $0 \leq p < q \leq n-1$, then $\IST(A,B,C)$ is finite.
	\end{itemize}
\end{lemma}
\begin{proof} \
	\begin{itemize}
		\item[$(1)$] Suppose that $\IST(A,B,C)= \infty$ and row $i$ and row $j$ in $A$ are $01^{n-1}$ and $1^{n-1}0$, respectively. Note that $B^{n-1}A$ gives $m$ consecutive rows in $M^2(A,B,C)$ obtained by applying the morphism to $1^{n-1}0$. Row $i$ in $B^{n-1}A$ is $x^{n-1}01^{n-1}$, where $x$ is row $i$ in $B$. Since $A$ is not a layered matrix, by Lemma~\ref{Lem-1res}, there is no $0$ in $x$. So $x^{n-1}01^{n-1}$ cannot be of the form $0^r1^s0^t$, $0^r(-1)^s0^t$ or $1^r0^s(-1)^t$. This contradicts to   Corollary~\ref{CorMainTheorem=>}.
			
		\item[$(2)$] Suppose that $\IST(A,B,C)= \infty$ and row $i$ and row $j$ in $A$ are $0(-1)^{n-1}$ and $(-1)^{n-1}0$, respectively. Note that $AC^{n-1}$ gives $m$ consecutive rows in $M^2(A,B,C)$ obtained by applying the morphism to $0(-1)^{n-1}$. Row $j$ in $AC^{n-1}$ is $(-1)^{n-1}0x^{n-1}$, where $x$ is row $j$ in $B$. Since $A$ is not a layered matrix, by Lemma~\ref{Lem-1res}, there is no $0$ in $x$. So $(-1)^{n-1}0x^{n-1}$ cannot be of the form $0^r1^s0^t$, $0^r(-1)^s0^t$ or $1^r0^s(-1)^t$. This contradicts to  Corollary~\ref{CorMainTheorem=>}.
			
		\item[$(3)$] Suppose that $\IST(A,B,C)= \infty$ and row $i$ and row $j$ in $A$ are $01^{n-1}$ and $(-1)^{n-1}0$, respectively. Note that $AB^{n-1}$ gives $m$ consecutive rows in $M^2(A,B,C)$ obtained by applying the morphism to $01^{n-1}$. Row $j$ in $AB^{n-1}$ is $(-1)^{n-1}0x^{n-1}$, where $x$ is row $j$ in $B$. Note that $(-1)^{n-1}0x^{n-1}$ must be of the form $0^r(-1)^s0^t$, and so $x=0^n$. Thus, $(-1)^{n-1}0x^{n-1} = (-1)^{n-1}0^{(n^2-n-1)}$ is a row in $M^2(A,B,C)$ having more than one $0$, which contradicts to Lemma~\ref{Lem-1res}.
			
		\item[$(4)$] \label{(4)} Suppose that $\IST(A,B,C)= \infty$ and row $i$ and row $j$ in $A$ are $0(-1)^{n-1}$ and $1^{n-1}0$, respectively. Note that $AC^{n-1}$ gives $m$ consecutive rows in $M^2(A,B,C)$ obtained by applying the morphism to $0(-1)^{n-1}$. Row $j$ in $AC^{n-1}$ is $1^{n-1}0x^{n-1}$, where $x$ is row $j$ in $C$. Since $A$ is not a layered matrix, by Lemma~\ref{Lem-1res}, there is no $0$ in $x$. Therefore, $1^{n-1}0x^{n-1}$ is of the form $1^r0^s(-1)^t$. So $x=(-1)^n$ and $1^{n-1}0x^{n-1} = 1^{n-1}0(-1)^{n^2-n}$ is a row in $M^2(A,B,C)$. Note that $B^{n-1}A$ gives $m$ consecutive rows in $M^2(A,B,C)$ obtained by application of the morphism to $1^{n-1}0$. Row $j$ in $B^{n-1}A$ is $y^{n-1}1^{n-1}0$, where $y$ is row $j$ in $B$. Since $A$ is not a layered matrix, by Lemma~\ref{Lem-1res}, there is no $0$ in $y$. Therefore, $y^{n-1}1^{n-1}0$ is of the form $0^r1^s0^t$. So $y=1^n$ and $y^{n-1}1^{n-1}0 = 1^{n^2-1}0$ is a row in $M^2(A,B,C)$.
		If every row of $M^2(A,B,C)$ has a $0$, then $1^{n-1}0(-1)^{n^2-n-1}$ and $1^{n^2-1}0$ break the second condition of Theorem~\ref{MainTheorem}. Hence, $G_o^2(A,B,C)$ is not semi-transitive and this leads to a contradiction. Suppose that $1^{\ell}(-1)^{n^2-\ell}$ is a row of $M^2(A,B,C)$ for some $0 \leq \ell \leq n^2$. By Lemma~\ref{Lem-no0row}, we have $G_o^2(A,B,C)$ is isomorphic to $G_o(N)$ where $N$ is a matrix obtained by deleting row $1^{\ell}(-1)^{n^2-\ell}$ from $M^2(A,B,C)$ and after that adding a zero column between the $\ell$-th and the $(\ell+1)$-th columns. So, $G_o(N)$ is semi-transitive. Note that $\ell$ must be $n-1$ or $n$, otherwise the row of $N$ obtained by adding a zero to $1^{n-1}0(-1)^{n^2-n-1}$ is not of the form $1^r0^s(-1)^t$. Then row of $N$ obtained by adding a zero to $1^{n^2-1}0$ in between the $\ell$-th and the $(\ell+1)$-th positions is not of the form $0^r1^s0^t$ or $0^r(-1)^s0^t$ or $1^r0^s(-1)^t$. Hence, by Corollary~\ref{CorMainTheorem=>}, $G_o(N)$ is not semi-transitive, which is a contradiction.

		\item[$(5)$] Suppose that $\IST(A,B,C)= \infty$ and row $i$ and row $j$ in $A$ are $1^p0(-1)^{n-p-1}$ and $1^q0(-1)^{n-q-1}$, respectively, where $0 \leq p < q \leq n-1$. Note that $B^pAC^{n-p-1}$ gives $m$ consecutive rows in $M^2(A,B,C)$ obtained by applying the morphism to $1^p0(-1)^{n-p-1}$. Row $i$  in $B^pAC^{n-p-1}$ is $x^{p}1^p0(-1)^{n-p-1}y^{n-p-1}$ where $x$ is row $i$ in $B$ and $y$ is row $i$ in $C$. Since $A$ is not a layered matrix, by Lemma~\ref{Lem-1res}, there is no more than one $0$ in any row of $M^2(A,B,C)$. By Corollar~\ref{CorMainTheorem=>}, we obtain $x^{p}1^p0(-1)^{n-p-1}y^{n-p-1}$ equals  $1^{np+p}0(-1)^{n^2-np-p-1}$. Note that $B^qAC^{n-q-1}$ gives $m$ consecutive rows in $M^2(A,B,C)$ obtained by application of the morphism to $1^q0(-1)^{n-q-1}$, and $x^{q}1^q0(-1)^{n-q-1}y^{n-q-1}$ is its row $i$. Similarly to the above, we have $x^{q}1^q0(-1)^{n-q-1}y^{n-q-1} = 1^{nq+q}0(-1)^{n^2-nq-q-1}$. That is, both $1^{np+p}0(-1)^{n^2-np-p-1}$ and $1^{nq+q}0(-1)^{n^2-nq-q-1}$ are rows of $M^2(A,B,C)$. If $p=0$ and $q=n-1$, then $0(-1)^{n-1}$ and $1^{n-1}0$ are rows in $A$ which is a contradiction by (4). So, one of $1^{np+p}0(-1)^{n^2-np-p-1}$ and $1^{nq+q}0(-1)^{n^2-nq-q-1}$ is of the form $1^r0^s(-1)^t$ for some $r,s,t>0$. If every row of $M^2(A,B,C)$ has a $0$, then $G_o^2(A,B,C)$ is not semi-transitive because the second condition of Theorem~\ref{MainTheorem} is not satisfied, and this is a contradiction. Suppose that $1^{\ell}(-1)^{n^2-\ell}$ is a row of $M^2(A,B,C)$ for some $0 \leq \ell \leq n^2$. By Lemma~\ref{Lem-no0row}, we have $G_o^2(A,B,C)$ is isomorphic to $G_o(N)$ where $N$ is a matrix obtained by deleting row $1^\ell(-1)^{n^2-\ell}$ from $M^2(A,B,C)$ and after that adding a zero column between the $\ell$-th and the $(\ell+1)$-th columns. So, $G_o(N)$ is semi-transitive. Note that $\ell$ must be $np+p$ or $np+p+1$, otherwise the row of $N$ obtained by adding a zero to $1^{np+p}0(-1)^{n^2-np-p-1}$ is not of the form $1^r0^s(-1)^t$. Similarly, $\ell$ must be $nq+q$ or $nq+q+1$, otherwise the row of $N$ obtained by adding a zero to $1^{nq+q}0(-1)^{n^2-nq-q-1}$ is not of the form $1^r0^s(-1)^t$. As $np+p$, $np+p+1$, $nq+q$ and $nq+q+1$ are all distinct, we have a contradiction

\end{itemize}
\vspace{-0.5cm}\end{proof}

	\begin{lemma} \label{Lem-4properties}
		Let $A$, $B$, $C$ be $m \times n$ matrices over $\{-1,0,1\}$ such that $A$ has a $0$ and $(A,B,C)$ is not independent from $B$ and $C$. Then,
		\begin{itemize}
			\item[$(1)$] if $1^p0(-1)^{n-p-1}$ and $01^{n-1}$ are rows in $A$, where $1 \leq p \leq n-2$, then $\IST(A,B,C)$ is finite; 
			\item[$(2)$] if $1^p0(-1)^{n-p-1}$ and $0(-1)^{n-1}$ are rows in $A$, where $1 \leq p \leq n-2$, then $\IST(A,B,C)$ is finite; 
			\item[$(3)$] if $1^p0(-1)^{n-p-1}$ and $1^{n-1}0$ are rows in $A$, where $1 \leq p \leq n-2$, then $\IST(A,B,C)$ is finite; 
			\item[$(4)$] if $1^p0(-1)^{n-p-1}$ and $(-1)^{n-1}0$ are rows in $A$, where $1 \leq p \leq n-2$, then $\IST(A,B,C)$ is finite.
		\end{itemize}
	\end{lemma}
	\begin{proof} \
	
		\begin{itemize}
			\item[$(1)$]  Suppose that $1^p0(-1)^{n-p-1}$ and $01^{n-1}$ are rows $i$ and $j$ in $A$, respectively, and $\IST(A,B,C)=\infty$. Note that $B^p A C^{n-p-1}$ gives $m$ consecutive rows in $M^2(A,B,C)$ obtained by applying the morphism to $1^p0(-1)^{n-p-1}$ in $M^1(A,B,C)$. Row $j$ in $B^p A C^{n-p-1}$ is $b^p01^{n-1}c^{n-p-1}$, where $b$ and $c$ are row $j$ in $B$ and $C$, respectively. So $b^p01^{n-1}c^{n-p-1}$ must be $0^r1^s0^t$ for some $r,s,t\geq 0$. Hence, $b=0^n$ and $c=1^n$. As $A$ is not a layered matrix, every row in $M^2(A,B,C)$ contains at most one $0$, which is a contradiction. Therefore, $\IST(A,B,C)<\infty$.
			
			\item[$(2)$]  This is given by (5) in Lemma~\ref{Lem-5properties}. 
			\item[$(3)$]  This is given by (5) in Lemma~\ref{Lem-5properties}. 
			\item[$(4)$]  Suppose that $1^p0(-1)^{n-p-1}$ and $(-1)^{n-1}0$ are rows $i$ and $j$ in $A$, respectively, and $\IST(A,B,C)=\infty$. Note that $B^p A C^{n-p-1}$ gives $m$ consecutive rows in $M^2(A,B,C)$ obtained by applying the morphism to $1^p0(-1)^{n-p-1}$ in $M^1(A,B,C)$. Row $j$ in $B^p A C^{n-p-1}$ is $b^p(-1)^{n-1}0c^{n-p-1}$, where $b$ and $c$ are row $j$ in $B$ and $C$, respectively. So, $b^p(-1)^{n-1}0c^{n-p-1}$ must be $0^r(-1)^s0^t$ for some $r,s,t\geq 0$. Hence, $b=(-1)^n$ and $c=0^n$. As $A$ is not a layered matrix, every row in $M^2(A,B,C)$ contains at most one $0$, which is a contradiction. Therefore, $\IST(A,B,C)<\infty$.
\end{itemize}
\vspace{-0.5cm}	
\end{proof}

\begin{definition}	Let $A,B,C$ be $m \times n$ matrices over $\{-1,0,1\}$. A triple $(A,B,C)$ is {\em left-$0$-invariant} if $A$, $B$, $C$ satisfy the following properties:
	\begin{itemize}
		\item every row in $A$ is in $\{01^{n-1}, 1^n, 0(-1)^{n-1}, (-1)^n\}$;
		\item every row in $B$ and $C$ is in $\{1^n,(-1)^n\}$;
		\item if $01^{n-1}$ appears as a row in $A$, then 
		\begin{itemize}
			\item row $i$ in $A$ is $01^{n-1}$ implies row $i$ in $B$ is $1^n$; 
			\item row $i$ in $A$ is $1^n$ implies row $i$ in $B$ is $1^n$;
			\item row $i$ in $A$ is $0(-1)^{n-1}$ implies row $i$ in $B$ is $(-1)^n$; 
			\item row $i$ in $A$ is $(-1)^n$ implies  row $i$ in $B$ is $(-1)^n$; 
		\end{itemize}
		\item if $0(-1)^{n-1}$ appears as a row in $A$, then 
		\begin{itemize}
			\item row $i$ in $A$ is $01^{n-1}$ implies row $i$ in $C$ is $1^n$; 
			\item row $i$ in $A$ is $1^n$ implies row $i$ in $C$ is $1^n$;
			\item row $i$ in $A$ is $0(-1)^{n-1}$ implies row $i$ in $C$ is $(-1)^n$; 
			\item row $i$ in $A$ is $(-1)^n$ implies row $i$ in $C$ is $(-1)^n$.
		\end{itemize}
	\end{itemize}
\end{definition}

\begin{definition}
	Let $A$, $B$, $C$ be $m \times n$ matrices over $\{-1,0,1\}$. A triple $(A,B,C)$ is {\em right-$0$-invariant} if $A$, $B$, $C$ satisfy the following properties:
	\begin{itemize}
		\item every row in $A$ is in $\{1^{n-1}0, 1^n, (-1)^{n-1}0,(-1)^n\}$;
		\item every row of $B$ and $C$ is in $\{1^n,(-1)^n\}$;
		\item if $1^{n-1}0$ appears as a row in $A$, then 
		\begin{itemize}
			\item row $i$ in $A$ is $1^{n-1}0$ implies row $i$ in $B$ is $1^n$; 
			\item row $i$ in $A$ is $1^n$ implies row $i$ in $B$ is $1^n$; 
			\item row $i$ in $A$ is $(-1)^{n-1}0$ implies row $i$in $B$ is $(-1)^n$; 
			\item row $i$ in $A$ is $(-1)^n$ implies row $i$ in $B$ is $(-1)^n$; 
		\end{itemize}
		\item if $(-1)^{n-1}0$ appears as a row in $A$, then 
		\begin{itemize}
			\item row $i$ in $A$ is $1^{n-1}0$ implies row $i$ in $C$ is $1^n$; 
			\item row $i$ in $A$ is $1^n$ implies row $i$ in $C$ is $1^n$; 
			\item row $i$ in $A$ is $(-1)^{n-1}0$ implies  row $i$ in $C$ is $(-1)^n$; 
			\item row $i$ in $A$ is $(-1)^n$ implies  row $i$ in $C$ is $(-1)^n$. 
		\end{itemize}
	\end{itemize}
\end{definition}

\begin{lemma} \label{lem-no0(-1)in-leftright-0-invariant}
	Let $A$, $B$, $C$ be $m \times n$ matrices over $\{-1,0,1\}$ such that $A$ has a $0$. Then,
	\begin{itemize}
		\item[$(1)$] If $(A,B,C)$ is left-$0$-invariant and  $01^{n-1} \notin R(A)$,\\ then $01^{n^k-1} \notin R^k(A,B,C)$ for any $k\geq 0$.
		\item[$(2)$] If $(A,B,C)$ is left-$0$-invariant and  $0(-1)^{n-1} \notin R(A)$,\\ then $0(-1)^{n^k-1} \notin R^k(A,B,C)$ for any $k>0$.
		\item[$(3)$] If $(A,B,C)$ is right-$0$-invariant and  $1^{n-1}0 \notin R(A)$,\\ then $1^{n^k-1}0 \notin R^k(A,B,C)$ for any $k>0$.
		\item[$(4)$] If $(A,B,C)$ is right-$0$-invariant and  $(-1)^{n-1}0 \notin R(A)$,\\ then $(-1)^{n^k-1}0 \notin R^k(A,B,C)$ for any $k>0$.
	\end{itemize}
\end{lemma}
\begin{proof}
	As all of the statements are proved in similar ways, we will only prove $(1)$. Assume $(A,B,C)$ is left-$0$-invariant and  $01^{n-1} \notin R(A)$. For $k=1$, it is obvious that $M^1(A,B,C)=A$ and then $01^{n-1} \notin R^1(A,B,C)$. Suppose $k \geq 2$ and $01^{n^k-1} \in R^k(A,B,C)$. Let $0x_1x_2 \cdots x_{n^{k-1}-1}$ be a row in $M^{k-1}(A,B,C)$ such that applying to it the morphism creates row $01^{n^k-1}$. That is, $01^{n^k-1}$ is a row in the matrix $AX_1X_2 \cdots X_{n^{k-1}-1}$, where $X_i \in \{A,B,C\}$, obtained from $0x_1x_2 \cdots x_{n^{k-1}-1}$ by application of the morphism. This is a contradiction because $01^{n-1} \notin R(A)$. Hence, $01^{n^k-1} \notin R^k(A,B,C)$.
\end{proof}

\begin{lemma} \label{lem-leftright-0-invariant}
	Let $A$, $B$, $C$ be $m \times n$ matrices over $\{-1,0,1\}$ such that $A$ has a $0$. If $(A,B,C)$ is left-$0$-invariant (resp., right-$0$-invariant), then $\IST(A,B,C)= \infty$.
\end{lemma}
\begin{proof}
	Suppose that $(A,B,C)$ is left-$0$-invariant. We will prove that for any $k>0$, $R^k(A,B,C) \subseteq \{01^{n^k-1}, 1^{n^k}, 0(-1)^{n^k-1}, (-1)^{n^k} \}$ by induction on $k$. From  the definition of a left-$0$-invariant tripple, we have that $R^1(A,B,C) = R(A) \subseteq \{01^{n-1}, 1^n, 0(-1)^{n-1}, (-1)^n \}$. Suppose $R^k(A,B,C) \subseteq \{01^{n^k-1}, 1^{n^k}, 0(-1)^{n^k-1}, (-1)^{n^k} \}$ for some $k$. If $01^{n-1} \notin R(A)$, then $0(-1)^{n-1} \in R(A)$ and, by Lemma~\ref{lem-no0(-1)in-leftright-0-invariant}, $01^{n^k-1} \notin R^k(A,B,C)$. So, every row in $M^k(A,B,C)$ is $1^{n^k}$, $0(-1)^{n^k-1}$ or $(-1)^{n^k}$. As every row in $M^{k+1}(A,B,C)$ is a row in an $m \times n^{k+1}$ matrix obtained by applying the morphism to a row in $M^k(A,B,C)$, we have that $$R^{k+1}(A,B,C)= R(B^{n^k}) \cup R(AC^{n^k-1}) \cup R(C^{n^k}).$$ We can see that $R(B^{n^k})$ and $R(C^{n^k})$ are subset of $\{1^{n^{k+1}}, (-1)^{n^{k+1}}\}$. Row $i$ in $AC^{n^k-1}$ is $1^{n^{k+1}}$, $(-1)^{n^{k+1}}$ and $0(-1)^{n^{k+1}-1}$ if row $i$ in $A$ is $1^n$, $(-1)^n$ and $0(-1)^n$, respectively. Hence, $R^{k+1}(A,B,C) \subseteq \{1^{n^{k+1}}, 0(-1)^{n^{k+1}-1}, (-1)^{n^{k+1}} \}$ in the case of $01^{n-1} \notin R(A)$. For the case of $0(-1)^{n-1} \notin R(A)$, we can follow similar arguments to see that $R^{k+1}(A,B,C) \subseteq \{01^{n^{k+1}-1}, 1^{n^{k+1}}, (-1)^{n^{k+1}} \}$. Assume both $01^{n-1}$ and $0(-1)^{n-1}$ are in $R(A)$. So, every row in $M^k(A,B,C)$ is $01^{n^k-1}$, $1^{n^k}$, $0(-1)^{n^k-1}$ or $(-1)^{n^k}$ and  $$R^{k+1}(A,B,C)=R(AB^{n^k-1}) \cup R(B^{n^k})  \cup R(AC^{n^k-1}) \cup R(C^{n^k}). $$ Note that $R(B^{n^k}), R(C^{n^k}) \subseteq \{1^{n^{k+1}}, (-1)^{n^{k+1}}\}$. Row $i$ in $AC^{n^k-1}$ is $1^{n^{k+1}}$, $(-1)^{n^{k+1}}$, $01^{n^{k+1}-1}$ and $0(-1)^{n^{k+1}-1}$ if row $i$ in $A$ is $1^n$, $(-1)^n$, $01^n$ and $0(-1)^n$, respectively. Row $i$ in $AB^{n^k-1}$ is $1^{n^{k+1}}$, $(-1)^{n^{k+1}}$, $01^{n^{k+1}-1}$ and $0(-1)^{n^{k+1}-1}$ if row $i$ in $A$ is $1^n$, $(-1)^n$, $01^n$ and $0(-1)^n$, respectively. Hence, $R^{k+1}(A,B,C) \subseteq \{1^{n^{k+1}}, 0(-1)^{n^{k+1}-1}, (-1)^{n^{k+1}} \}$. Thus, we have shown that, for any $k>0$, $$R^k(A,B,C) \subseteq \{01^{n^k-1}, 1^{n^k}, 0(-1)^{n^k-1}, (-1)^{n^k} \}.$$ By Corollary~\ref{corr-MainTheorem}, $G_o^k(A,B,C)$ is semi-transitive for any $k>0$, which means that $\IST(A,B,C)= \infty$.
\end{proof}

\begin{theorem} \label{thm-notIndeoendent-A=NL-B,C=L}
	Let $A$, $B$, $C$ be $m \times n$ matrices over $\{-1,0,1\}$ such that $A$ has a $0$ and $(A,B,C)$ is not independent from $B$ and $C$. Suppose $A$ is not a layered matrix but $B$ and $C$ are layered matrices. Then, $\IST(A,B,C)= \infty$ if and only if one of the following conditions holds:
	\begin{itemize}
		\item $(A,B,C)$ is left-$0$-invariant.
		\item $(A,B,C)$ is right-$0$-invariant.
		\item $R(A)=\{1^p01^{n-p-1}\}$
		for some $p \in \{1,2,\ldots, n-2\}$, 
		and $B$ and $C$ are all $1$ and $(-1)$ matrices, respectively.
	\end{itemize}
\end{theorem}
\begin{proof}
	Assume $\IST(A,B,C)=\infty$. Since $A$ is not a layered matrix, by Lemma~\ref{Lem-1res}, every row of $M^1(A,B,C)=A$ contains at most one $0$. Then, every row in $A$ is $01^{n-1}$, $1^{n-1}0$, $1^n$, $0(-1)^{n-1}$, $(-1)^{n-1}0$, $(-1)^n$,  $1^p0(-1)^{n-p-1}$ or $1^q(-1)^{n-q}$ for some $p \in \{1,2, \ldots , n-2\}$ and $q \in \{1,2, \ldots , n-1\}$. Since $(A,B,C)$ is not independent from $B$ and $C$, and $B$ and $C$ are layered matrices, we have every row of $B$ and $C$ must be $1^n$ or $(-1)^n$, otherwise there is a row in $M^k(A,B,C)$ having more than one $0$ for some $k$.
	
	If $1^q(-1)^{n-q}$ is row $i$ in $A$ for $q \in \{1,2, \ldots , n-1\}$, then $G_o^1(A,B,C)$ is isomorphic to $G(N_1)$ where $N_1$ is an $(m-1) \times (n+1)$ matrix obtained by deleting row $i$ and then adding a zero column between the $q$-th and $(q+1)$-th columns. Note that $G(N_1)$ is semi-transitive. By Corollary~\ref{CorMainTheorem=>}, every row of $N_1$ is of the form $ 0^r  1^s  0^t $ or $ 0^r  (-1)^s  0^t $ for some non-negative integers $r$, $s$, $t$. So every row of $A$ except row $i$ is $1^q(-1)^{n-q}$, $1^{q-1}0(-1)^{n-q}$ or $1^q0(-1)^{n-q-1}$. By (5) in Lemma~\ref{Lem-5properties}, we have that $A$ cannot contain both $1^{q-1}0(-1)^{n-q}$ and $1^q0(-1)^{n-q-1}$ as its rows. If $1^{q-1}0(-1)^{n-q} \notin R(A)$, then  $$A=\begin{bmatrix}
		1^q		&	a_1	&	(-1)^{n-q-1}	\\
		1^q		&	a_2	&	(-1)^{n-q-1}	\\
		\vdots	&\vdots	&	\vdots		\\
		1^q		&	a_m	&	(-1)^{n-q-1}
		\end{bmatrix}$$ 
	for $a_i \in \{0,1\}$. Since $A$ has a $0$, there is row $j$ in $A$ of the form $1^q0(-1)^{n-q-1}$. Let $b$ and $c$ be row $j$ in $B$ and $C$, respectively. Note that $B^qAC^{n-q-1}$ is $m$ consecutive rows of $M^2$ obtained by applying the morphism to $1^q0(-1)^{n-q-1}$. Then, $b^q1^q0(-1)^{n-q-1}c^{n-q-1}$ is row $j$ in $M^2(A,B,C)$ and it must be of the form $1^r0^s(-1)^t$ for some $r,s,t\geq 0$. So, we obtain $b=1^n$ and $c^n=(-1)^n$ and $1^{nq+q}0(-1)^{n^2-nq-q-1}$ is a row of $M^2(A,B,C)$. Note that $B^qC^{n-q}$ is $m$ consecutive rows of $M^2$ obtained by applying the morphism to $1^q(-1)^{n-q}$, and $1^{nq}(-1)^{n(n-q)}$ is row $j$ in $B^qC^{n-q}$. Then, $G_o^2(A,B,C)$ is isomorphic to $G(N_2)$ where $N_2$ is a matrix obtained by deleting the row $1^{nq}(-1)^{n(n-q)}$ and then adding a zero column between the $(nq)$-th and $(nq+1)$-th columns. So, $1^{nq}01^{q}0(-1)^{n^2-nq-q-1}$ is a row in $N_2$. Therefore, by Corollary~\ref{CorMainTheorem=>}, we have $N_2$ is not semi-transitive which contradicts to semi-transitivity of $G_o^2(A,B,C)$. By the same argument, we also obtain a contradiction in the case of $1^{q-1}0(-1)^{n-q} \notin R(A)$. Hence $1^q(-1)^{n-q}$ cannot be a row in $A$.

	Suppose that $1^p0(-1)^{n-p-1}$ is row $i$ in $A$.
	By Corollaries~\ref{010 and 10-1} and \ref{0-10 and 10-1}, we have that $1^{n}$ and $(-1)^{n}$ are not rows in $M^1(A,B,C)$. By Lemma~\ref{Lem-4properties}, we have that $01^{n-1}$, $1^{n-1}0$, $0(-1)^{n-1}$ and $(-1)^{n-1}0$ are not rows in $M^1(A,B,C)$. If there is a row in $A$ of the form $1^{u}0(-1)^{n-u-1}$, where $1 \leq u \leq n-2$, by (5) in Lemma~\ref{Lem-5properties}, we have $p=u$. Hence, we obtain $$A= \begin{bmatrix}
	1^p		&	0	&	(-1)^{n-p-1}	\\
	1^p		&	0	&	(-1)^{n-p-1}	\\
	\vdots	&\vdots	&	\vdots		\\
	1^p		&	0	&	(-1)^{n-p-1}
	\end{bmatrix} \text{~~where~~} 1 \leq p < n-2.$$
	Let $b$ and $c$ be row $j$ in $B$ and $C$, respectively, for any $1 \leq j \leq m$. Note that $B^pAC^{n-p-1}$ is $m$ consecutive rows of $M^2$ obtained by applying the morphism to $1^p0(-1)^{n-p-1}$. Then, $b^p1^p0(-1)^{n-p-1}c^{n-p-1}$ is row $j$ in $M^2(A,B,C)$ and it must be of the form $1^r0^s(-1)^t$ for some $r,s,t\geq 0$. So, we obtain $b=1^n$ and $c=(-1)^n$. Hence, we see that $B$ and $C$ are all $1$ matrix and all $(-1)$ matrix, respectively.
	
	Assume that $1^p0(-1)^{n-p-1}$ is not a row in $A$ for any $1 \leq p \leq n-2$. That is, every row in $A$ is $01^{n-1}$, $1^{n-1}0$, $1^n$, $0(-1)^{n-1}$, $(-1)^{n-1}0$ or $(-1)^n$. By Lemma~\ref{Lem-5properties}, we need to consider the following two cases. \\

\noindent
${\bf Case 1}$: $01^{n-1},0(-1)^{n-1} \in R(A)$ and $1^{n-1}0,(-1)^{n-1}0 \notin R(A)$. That is, every row in $A$ is $01^{n-1}$, $1^n$, $0(-1)^{n-1}$ or $(-1)^n$. Suppose that $01^{n-1}$ is a row in $A$. Then, $AB^{n-1}$ is $m$ consecutive rows in $M^2(A,B,C)$. Let row $i$ in $B$ be $b$. Consider the following subcases: 
		\begin{itemize}
			\item If row $i$ in $A$ is $01^{n-1}$, then $01^{n-1}b^{n-1}$ is a row in $M^2(A,B,C)$. Since $b \neq 0^n$, we have $b=1^n$.
			\item If row $i$ in $A$ is $1^{n}$, then $1^{n}b^{n-1}$ is a row in $M^2(A,B,C)$. Since $b \neq 0^n$, we have $b=1^n$.
			\item If row $i$ in $A$ is $0(-1)^{n-1}$, then $0(-1)^{n-1}b^{n-1}$ is a row in $M^2(A,B,C)$. Since $b \neq 0^n$, we have $b=(-1)^n$.
			\item If row $i$ in $A$ is $(-1)^{n}$, then $(-1)^{n}b^{n-1}$ is a row in $M^2(A,B,C)$. Since $b \neq 0^n$, we have $b=(-1)^n$.
		\end{itemize}
		 Suppose that $0(-1)^{n-1}$ is a row in $A$. Then, $AC^{n-1}$ is $m$ consecutive rows in $M^2(A,B,C)$. Let row $i$ in $C$ be $c$. Consider the following subcases: 
		\begin{itemize}
		\item If row $i$ in $A$ is $01^{n-1}$, then $0(-1)^{n-1}c^{n-1}$ is a row in $M^2(A,B,C)$. Since $c \neq 0^n$, we have $c=1^n$.
		\item If row $i$ in $A$ is $1^{n}$, then $1^{n}c^{n-1}$ is a row in $M^2(A,B,C)$. Since $c \neq 0^n$, we have $c=1^n$.
		\item If row $i$ in $A$ is $0(-1)^{n-1}$, then $0(-1)^{n-1}c^{n-1}$ is a row in $M^2(A,B,C)$. Since $c \neq 0^n$, we have $c=(-1)^n$.
		\item If row $i$ in $A$ is $(-1)^{n}$, then $(-1)^{n}c^{n-1}$ is a row in $M^2(A,B,C)$. Since $c \neq 0^n$, we have $c=(-1)^n$.
	\end{itemize}
	Thus, we see that $(A,B,C)$ is left-$0$-invariant. \\

\noindent
${\bf Case 2}$: $1^{n-1}0,(-1)^{n-1}0 \in R(A)$ and $01^{n-1},0(-1)^{n-1} \notin R(A)$. With the same way of the case 1, we can prove that $(A,B,C)$ is right-$0$-invariant. \\

	Thus, ``$\Rightarrow$'' has been proved. Lemma~\ref{lem-leftright-0-invariant} gives us the converse. 
\end{proof}

\section{Direction of further research}

In this paper, we fully classified semi-transitivity of infinite families of directed split graphs generated by iterations of morphisms in the cases when the matrix $A$ has a 0. This research is a first step towards a classification of semi-transitive directed graphs in terms of positions of $0$s and $1$s (and $(-1)$s in the lower-triangular case) in the adjacency matrices. An application of such a classification could be in finding more efficient algorithms to recognize semi-transitivity of a directed graph, which is a problem solvable in polynomial time~\cite{KL15}. More importantly, a classification of semi-transitive directed graphs via adjacency matrices may lead to a better understanding of which (undirected) graphs admit semi-transitive orientations; this is an NP-complete problem~\cite{K17,KL15}.  Should the general problem resist attempts to solve it, one could shift their attention to classification of semi-transitivity of naturally defined (infinite) families of directed graphs. Such a shift should allow discovering new methods to deal with semi-transitivity of oriented graphs, and hence bring us closer to solving the general problem. 

For yet another direction of research, note that Definition~\ref{IWR-def} of the index of semi-transitivity $\IST(A,B,C)$ makes sense in many situations when $A$ has no $0$'s. For example, if $A$, $B$ and $C$ contain only $1$'s, we still can apply Definition~\ref{IWR-def} to see that $\IST(A,B,C)=\infty$. On the other hand, Definition~\ref{IWR-def} does not work, for example, in the case when $A$ is any matrix without $0$'s while $B$ and $C$ contain only $0$'s, as the infinite graph $G_o(A,B,C)$ is then not well-defined. Indeed, in the later case we see that  $G_o^i(A,B,C)$ is not an induced subgraph of $G_o^{i+1}(A,B,C)$ while $G_o^i(A,B,C)$ is an induced subgraph of $G_o^{i+2}(A,B,C)$ for any $i\geq 0$, so that we have two infinite chains of induced subgraphs leading to two different infinite graphs as the limits (one of which is with no edges between the clique and the independent set). For another example, letting $A$ be an all one matrix, $B$ be an all $(-1)$ matrix, and $C$ be an all zero matrix, we witness the situation of three infinite chains of induced subgraphs with three infinite graphs as the limits. 

In any case, the problem we solved in this paper can be extended to the case of  matrices $A$ with no $0$'s in the situations when the limiting infinite graph is uniquely defined, and the goal then is to classify such triples $(A,B,C)$ with $\IST(A,B,C)=\infty$. Of course, extra care should be taken about Definition~\ref{IWR-def} as it still may not work. For example, $A$ without $0$'s can easily be chosen so that $G_0^1(A,B,C)$ has directed cycles and thus is not semi-transitive, while then choosing $B$ and $C$ be all one matrices, we see that $G_0^k(A,B,C)$ is semi-transitive for $k>1$, so that the limiting graph is also semi-transitive and it is natural to assume that  $\IST(A,B,C)=\infty$, while by Definition~\ref{IWR-def}, $\IST(A,B,C)=1$. However, natural adjustments to Definition~\ref{IWR-def} could be introduced. For example, we can define  $\IST(A,B,C) := \infty$ if there exists a  natural number $k$ such that $G^i_o(A,B,C)$ is semi-transitive for every $i \geq k$.


\begin{thebibliography}{20}
\bibitem{CKS19} H. Z.Q. Chen, S. Kitaev, A. Saito. Representing split graphs by words, {\em Discussiones Mathematicae Graph Theory}, to appear 2021.
\bibitem{CJKS2020} P. Choudhary,  P. Jain, R. Krithika, V. Sahlot. Vertex deletion on split graphs: beyond 4-hitting set. {\em Theoret. Comput. Sci.} {\bf 845} (2020), 21--37. 
\bibitem{FH77} S. Foldes, P. L. Hammer. Split graphs. In Proceedings of the {\em Eighth Southeastern Conference on Combinatorics, Graph Theory and Computing} (Louisiana State Univ., Baton Rouge, 1977), pages 311--315. Congressus Numerantium, No. XIX.
\bibitem{HKP11} M. Halld\'orsson, S. Kitaev, A. Pyatkin. Alternation graphs, Lecture Notes in Computer Science {\bf 6986} (2011) 191--202. Proceedings of the 37th International Workshop on Graph-Theoretic Concepts in Computer Science, WG 2011, Tepla Monastery, Czech Republic, June 21--24, 2011.
\bibitem{HKP16} M. Halld\'orsson, S. Kitaev, A. Pyatkin. Semi-transitive orientations and word-representable graphs. {\em Discr. Appl. Math.} {\bf 201} (2016) 164--171.
\bibitem{I2021} K. Iamthong. Word-representability of split graphs generated by morphisms. arXiv:2104.14872, 2021.
\bibitem{K17} S. Kitaev. A Comprehensive Introduction to the Theory of Word-Representable Graphs, In: Developments in Language Theory: 21st International Conference, DLT, Liege, Aug 7--11, 2017. {\em Lecture Notes in Comp. Sci.} {\bf 10396} (2017) 36--67. 
\bibitem{KL15} S. Kitaev, V. Lozin. Words and Graphs, {\em Springer}, 2015.
\bibitem{KLMW17} S. Kitaev, Y. Long, J. Ma, H. Wu. Word-representability of split graphs, {\em Journal of Combinatorics}, to appear.
\bibitem{Loth} M. Lothaire. {\em Combinatorics on words}, Encyclopedia of Mathematics and its Applications, 17, Addison-Wesley Publishing Co., Reading, Mass, 1983.
\end{thebibliography}
\end{document}